\documentclass[10pt]{article}
\usepackage{amsmath}
\usepackage{graphicx}
\usepackage{cite}
\usepackage{authblk}
\pdfoutput=1
\usepackage{amsfonts}
\usepackage{float}
\newcommand{\Mod}[1]{\ (\text{mod}\ #1)}
\usepackage{mathtools, cuted}

\usepackage{lipsum, color}
\usepackage{amssymb}
\usepackage[margin=1.8cm]{geometry}
\usepackage{amsthm}
\author[1]{Carl P. Dettmann\thanks{Email: Carl.Dettmann@bris.ac.uk}} 
\author[1]{Vitaly Fain\thanks{Email: vf13950@bristol.ac.uk}}
\author[2,3]{Dmitry Turaev\thanks{Email: dturaev@imperial.ac.uk}}
\affil[1]{University of Bristol, School of Mathematics, University Walk, Bristol BS8 1TW, UK}
\affil[2]{Imperial College London, SW7 2AZ, UK}
\affil[3]{ Lobachevsky University of Nizhny Novgorod, Gagarina 23, 603950 Nizhny Novgorod, Russia}
\usepackage{appendix} 

\newtheorem{theorem}{Theorem}[section]
\newtheorem{lemma}[theorem]{Lemma}
\newtheorem{proposition}[theorem]{Proposition}

\newtheorem{remark}[theorem]{Remark}
\DeclareMathOperator{\sn}{sn}
\DeclareMathOperator{\cd}{cd}
\DeclareMathOperator{\dn}{dn}
\DeclareMathOperator{\cn}{cn}
\DeclareMathOperator{\e}{e}

\DeclareMathOperator{\sech}{sech} 
\title{Splitting of separatrices, scattering maps, and energy growth for a billiard inside a time-dependent 
symmetric domain close to an ellipse.}

\begin{document}

\maketitle

\begin{abstract}  We study billiard dynamics inside an ellipse for which the axes lengths 
are changed periodically in time and an $O(\delta)$-small quartic polynomial deformation is added to the boundary.
In this situation the energy of the particle in the billiard is no longer conserved. We show a Fermi acceleration 
in such system: there exists a billiard trajectory on which the energy tends to infinity. The construction 
is based on the analysis of dynamics in the phase space near a homoclinic intersection 
of the stable and unstable manifolds of the normally hyperbolic invariant cylinder $\Lambda$, 
parameterised by the energy and time, that corresponds to the motion along the major axis of the ellipse. 
The proof depends on the reduction of the billiard map near the homoclinic channel to an iterated function system 
comprised by the shifts along two Hamiltonian flows defined on $\Lambda$. The two flows approximate 
the so-called inner and scattering maps, which are basic tools that arise in the studies of the Arnold diffusion; 
the scattering maps defined by the projection along the strong stable and strong unstable foliations $W^{ss,uu}$ 
of the stable and unstable invariant manifolds $W^{s,u}(\Lambda)$ at the homoclinic points. Melnikov type calculations
imply that the behaviour of the scattering map in this problem is quite unusual: it is only defined on a small 
subset of $\Lambda$ that shrinks, in the large energy limit, to a set of parallel lines $t=const$ as $\delta\to 0$.
\end{abstract}

\section{Introduction and main results}

Billiards are Hamiltonian dynamical systems, representing the motion of a point particle inside a domain 
$Q$ (the billiard table) in a straight line with constant speed and elastically bouncing off the boundary
of the domain $\partial {Q}$. The study of billiard systems was initiated by Birkhoff \cite{birkhoff1927periodic}. 
Depending on the boundary shape, billiard's dynamical behaviour may range from completely integrable to chaotic. 
The billiard inside an ellipse is the only known integrable strictly convex billiard \cite{avila2014integrable}. 
The integrability of elliptic billiard is closely connected to the existence of a continuous family of caustics. 
A caustic is a curve such that if a billiard trajectory segment is tangent to it, all other segments of the trajectory 
are also tangent to the same curve. For the trajectories that do not intersect the segment connecting the foci of the 
ellipse, the caustics are confocal ellipses while for the trajectories that intersect this segment the caustics 
are confocal hyperbolas. The period two trajectory along the major axis is hyperbolic, with stable and unstable 
manifolds that coincide. The corresponding motions repeatedly go  through the foci and converge to the major axis.

Billiards with time-dependent boundaries have received much attention in recent years \cite{koiller1995time}. 
One of the fundamental issues here is determining whether the particle energy may grow without bound as a result 
of repeated elastic collisions with the moving boundary. This phenomenon is called Fermi acceleration, after 
Fermi who first proposed it in his studies of highly energetic cosmic rays \cite{fermi1949mesons}. The existence 
of Fermi acceleration has been investigated theoretically and numerically in various billiard geometries. 
The simplest one-dimensional case corresponding to a particle bouncing between two periodically moving walls
(Fermi-Ulam model) and its variants is already very subtle and the existence of Fermi acceleration has been shown
to depend on the class of smoothness of the motion of the wall \cite{pustyl1983ulam, kruger1995acceleration, zharnitsky1998instability}.

For domains in two dimensions and higher, it has been observed 
\cite{koiller1995time, loskutov2000properties, loskutov1999mechanism, kamphorst1999bounded, de2006fermi, shah2010exponential}  
that the acceleration depends on the structure of the phase space of the static ''frozen" billiard. It has been 
conjectured by Losktutov, Ryabov and Akinshin (LRA) \cite{loskutov2000properties} and consequently proved 
in \cite{gelfreich2008fermi} that a sufficient condition for Fermi acceleration is the presence of a Smale horseshoe
in the phase space of the frozen billiard.  On the other hand, it has been shown \cite{kamphorst1999bounded} that
the energy of trajectories in the time-dependent circle billiard stays bounded due to the angular momentum conservation.
Earlier Koiler et al. \cite{koiller1996static} numerically studied time-dependent perturbations of elliptic billiards
and did not observe sustained energy growth. However, more detailed numerical simulations by 
Lenz et al. \cite{lenz2008tunable, lenz2009evolutionary, lenz2010phase} showed slow growth of the particle speed 
when initial conditions belong to the separatrix region. An elliptic billiard with a slow boundary perturbation
and a slow angular velocity was also studied by Itin and Neishtadt \cite{itin2003resonant} who investigated 
the destruction of adiabatic invariants near a separatrix. In this paper we further push the study of Fermi
acceleration in periodically perturbed ellipse.

Fermi acceleration question is a part of the general question of energy growth in a priori unstable Hamiltonian systems,
that also includes Mather acceleration problem \cite{bolotin1999unbounded, delshams2008geometric, gelfreich2008unbounded}. 
Since  time-dependent billiards on a plane are given by a nonautonomous Hamiltonian systems with two and a half degrees of 
freedom, the billiard map is exact symplectic four-dimensional diffeomorphism \cite{gelfreich2012fermi, koiller1995time}.
In particular, invariant KAM-tori, if exist, do not divide the phase space into invariant regions 
and Arnold diffusion \cite{arnold1964instability} may occur. Arnold diffusion refers to the instability 
of action variables in a Hamiltonian system with $n>2$ degrees of freedom of the form 
$H = H_{0}(I) + \epsilon H_{1}(I, \varphi, \epsilon)$ where $(I, \varphi)$ are action-angle variables, $\epsilon$ 
is small, and $H_{0}$ is integrable. Following terminology in \cite{chierchia1994drift}, a Hamiltonian system 
is called a-priori unstable if the integrable part $H_{0}$ has a normally hyperbolic invariant 
manifold $\Lambda$ with stable and unstable manifolds $W^{s,u}(\Lambda)$ that coincide in a homoclinic loop. 
Under small perturbations, $\Lambda$ and $W^{s,u}(\Lambda)$ persist but $W^{s,u}(\Lambda)$ may intersect transversally
along a homoclinic set $\Gamma$. In this case the diffusing orbit stays near $\Lambda$ most of the time, 
occasionally making a trip near $\Gamma$. It was shown by Treschev \cite{treschev2004evolution} and Delshams, de la LLave and Seara
\cite{Delshams2006} that such homoclinic excursions can lead to a systematic drift of the action variable in the
a priori unstable case.

The technique for the analysis of such excursions, which is also used in this paper, goes back to 
the works of Delshams, de la LLave and Seara \cite{delshams2000geometric, delshams2008geometric} where
notions of the inner and scattering maps have been introduced and studied in detail. The inner map is 
the restriction of the dynamics 
to $\Lambda$, and the scattering map relates two points on $\Lambda$ that are connected asymptotically 
in the past and future if the intersection of $W^{s,u}(\Lambda)$ is 
\textit{strongly transverse} \cite{gelfreich2017arnold} along $\Gamma$. The iteration function system obtained 
by successive application of these two maps in an arbitrary order gives the diffusing orbit 
if they do not have common invariant curves \cite{le2007drift, moeckel2002generic, nassiri2012robust, gelfreich2017arnold}.
It was shown e.g. in \cite{gelfreich2017arnold} that the finite-length diffusing orbits of the iterated 
function system on $\Lambda$ correspond to Arnold diffusion in the original diffeomorphism near $\Lambda\cup \Gamma$, 
under the assumption of strong transversality of homoclinic intersections. This result was generalised by 
Gidea, de la Llave and Seara \cite{gidea2014general} to the orbits of semi-infinite length.

In this paper we study the time-dependent four-dimensional billiard map $B$ (defined in Section 3) 
describing the motion of a billiard inside the planar domain with the time-dependent boundary
\begin{equation}
\frac{x^2}{a^{2}(t)} + \frac{y^2}{b^{2}(t)} = 1 +  \frac{2 \delta y^{4}}{b^{4}(t)},
\label{ovalss}
\end{equation}
where $0 < | \delta | \ll 1$ is a constant parameter and $0<b(t)<a(t)$ are periodic $C^{r+1}$-smooth functions 
($r \geq 4$) of time (so $t \in \mathbb{S}^{1}$), and $(x,y) \in \mathbb{R}^{2}$. The boundary (\ref{ovalss}) 
may be viewed as an ellipse with time-periodically changing semi-axes $a(t)$ and $b(t)$ plus  an $O(\delta)$ 
quartic polynomial perturbation superimposed at each fixed value of $t$. Let $\mathcal{E}(t)$ be the kinetic energy
of the particle in the billiard bounded by (\ref{ovalss}). We prove here the following

\begin{theorem}
Let $0 < b(t) < a(t)$ be time-periodic $C^{r+1}$-functions ($r \geq 4$) such that 
the function $\frac{a(t)}{b(t)}$ has a nondegenerate critical point. Then, there exists a constant $C>0$, 
independent of $\delta$, such that for any $\mathcal{E}_{0} \geq \frac{C}{|\delta|}$, there exists 
a billiard trajectory for which the energy $\mathcal{E}(t)$ grows from $\mathcal{E}_{0}$ to infinity.
\end{theorem}

We note that if one replaces the $O(\delta)$ quartic polynomial perturbation in the right-hand side of (\ref{ovalss})
by another $O(\delta)$ perturbation that also destroys integrability of the static frozen ellipse for every fixed $t$ 
(for instance symmetric entire perturbations studied in \cite{delshams1996poincare}), then Theorem 1.1 should still hold.
We however restrict ourselves to a particular form of the perturbation, in order to keep the computations explicit.
We do not know whether the measure of the set of orbits for which the energy grows to infinity is positive. However, the same construction we use in the proof can show the existence of ``diffusing'' orbits which take 
every sufficiently large value of energy infinitely many times, following an arbitrary given itinerary,
so the evolution of energy is sensitive to initial conditions and has to be described by some random process.

The proof of Theorem 1.1 is based on the study of inner and scattering maps and an application of the theory 
developed in \cite{gelfreich2017arnold}.  The scheme of the proof is as follows.

Each time the particle hits the boundary of (\ref{ovalss}), one records the collision time moment $t\in \mathbb{S}^{1}$,
the particle kinetic energy $\mathcal{E}$, the angular variable $\varphi$ that determines the position of the collision
point on the billiard boundary, and the post-collision reflection angle $\theta$. Then, the particle motion is described
by the billiard map $B$ in the four-dimensional space of variables $(\varphi, \theta, \mathcal{E}, t)$. We assume that
the speed $w  = \sqrt{2\mathcal{E}}$ of the particle is large compared to the speed with which the boundary moves. This
invokes the presence of two time scales in the problem: the variables $\mathcal{E}$ and $t$ vary slowly, 
while $(\varphi, \theta)$ change fast. To make the presence of different time scales more apparent, we scale variables
like it was done in \cite{bolotin1999unbounded}: take large 
speed $w^{*}$, introduce a small parameter $\varepsilon=\frac{1}{w^{*}}$ and the rescaled energy 
$E=\varepsilon^2 \mathcal{E}$. Then, for any bounded interval of $E$, the map $B$ becomes near-identity 
in terms of $(E,t)$, i.e., it becomes $\varepsilon$-close to the two-dimensional billiard map $B_{s}$ 
corresponding to static boundary (\ref{ovalss}) at fixed $t$, with augmented phase space to account for $(E,t)$.
Hence, the map $B$ may be expanded in series of $\varepsilon=\frac{1}{w^{*}}$ and $\delta$. We will give full details
of this construction in section 3. 

If the particle moves along the major axis, it will never leave the major axis. This motion corresponds to an invariant
manifold $\Lambda$ in the phase space of the billiard map $B$, a two-dimensional cylinder parametrised by 
$(\mathcal{E},t)$. In the static billiard, the motion along the major axis is a saddle periodic orbit
for each frozen value of $\mathcal{E}$ and $t$. Therefore, because the map $B$ is close so the static billiard map
in the rescaled variables, it follows that the cylinder $\Lambda$ is normally hyperbolic. In particular, it has 
three-dimensional stable and unstable manifolds $W^{s,u}(\Lambda)$ foliated by the strong stable and unstable foliations
$W^{ss,uu}(\Lambda)$. These geometric objects are inherited by $B$ from the stable and unstable separatrices
of the static billiard's motion along the major axis. 

The restriction of $B$ to $\Lambda$ is close to identity (when written in the coordinates $(E,t)$ where $E$
is the rescaled energy). It is well-known \cite{benettin1994hamiltonian,neishtadt1984separation} that a near-identity
$C^{l}$-smooth (or analytic) symplectic map $x_{1} = x_{0} + \nu f(x_{0})$ with small $\nu$ is approximated by
a time-$\nu$ shift along the orbits of an autonomous Hamiltonian system up to accuracy $O(\nu^{l+1})$
(or exponential accuracy in $\nu$ for analytic case). Therefore, the map 
$B \mid _{\Lambda}$ (which we call the inner map) is approximated to a high level of accuracy by 
the time-shift along level curves of a certain Hamiltonian $H_{in}$. 

As we mentioned, the billiard in ellipse is integrable, so the stable and unstable manifolds of the periodic orbit
that corresponds to the motion along the major axis coincide. It is well-known that the resulting separatrix surface
is a graph of a function $\theta$ of $\varphi$, where $\varphi\in(0,\pi)$. Therefore, 
for any small $\beta>0$, the (perturbed) stable and unstable manifolds $W^{s,u}(\Lambda)$
at sufficiently small $\varepsilon$ and $\delta$ can be expressed as 
graphs $\theta=\theta^{s,u}_{\varepsilon, \delta}(\varphi, E, t)$ (see section 4.3) 
over the interval $\varphi \in (\beta, \pi - \beta)$. They are $O\left(\varepsilon, \delta\right)$-close
to the unperturbed manifolds $\theta=\theta^{s,u}_{0,0}$, therefore they can be expanded in series 
of $\varepsilon, \delta$. The zeroes of the difference (see section 4.3) 
$\theta^{s}_{\varepsilon, \delta} - \theta^{u}_{\varepsilon, \delta}$  correspond to the primary 
homoclinic intersections $(\varphi_{0}, \theta_{0}, \mathcal{E}_{0}, t_{0}) \in \Gamma $ where $\Gamma$
is the homoclinic set. If the corresponding unstable leaf of the foliation $W^{uu}(\Lambda)$ intersects
transversely the stable manifold $W^{s}(\Lambda)$ at the homoclinic point
$(\varphi_{0}, \theta_{0}, \mathcal{E}_{0}, t_{0})$, the intersection is called strongly transversal (see section 4).
When strong transversality condition is satisfied, projecting from the homoclinic point to $\Lambda$ along 
the corresponding unstable fiber of $W^{u}(\Lambda)$ and stable fiber of $W^{s}(\Lambda)$ produces a pair of points
in $\Lambda$ which are related by what is called the scattering map $S_{\Gamma}$. Its domain of definition 
$\bar{\Lambda} \subset \Lambda$ is the projection of the set of strong-transverse primary homoclinic points 
by the strong unstable fibers; the image $S_{\Gamma} (\bar{\Lambda})\subset \Lambda$ is the projection of 
the set of primary homoclinic points by the strong stable fibers. 
 
The scattering map is exact symplectic \cite{delshams2008geometric,gelfreich2017arnold}. For any bounded interval
of the rescaled energy $E$, this map
is close to identity, so it is well approximated by the time-$\varepsilon$ shift along the level curves
of a Hamiltonian $H_{out}$. We build the trajectory whose energy grows to infinity by following level 
curves of $H_{in}$ and $H_{out}$, and switching to the level curve which leads to the larger energy gain
in the immediate future. This construction is similar to \cite{gelfreich2008unbounded}, however the application
of inner and scattering maps is novel. Formally, we define $H_{in}$ and $H_{out}$ everywhere on $\Lambda$ but 
the switch of the orbit of $B$ to the level curve of $H_{out}$ is only allowed at the domain of 
definition of the scattering map. We find that this domain has a non-trivial structure in our problem.

\begin{figure}[h]
    \centering
    \includegraphics[width=0.4\textwidth]{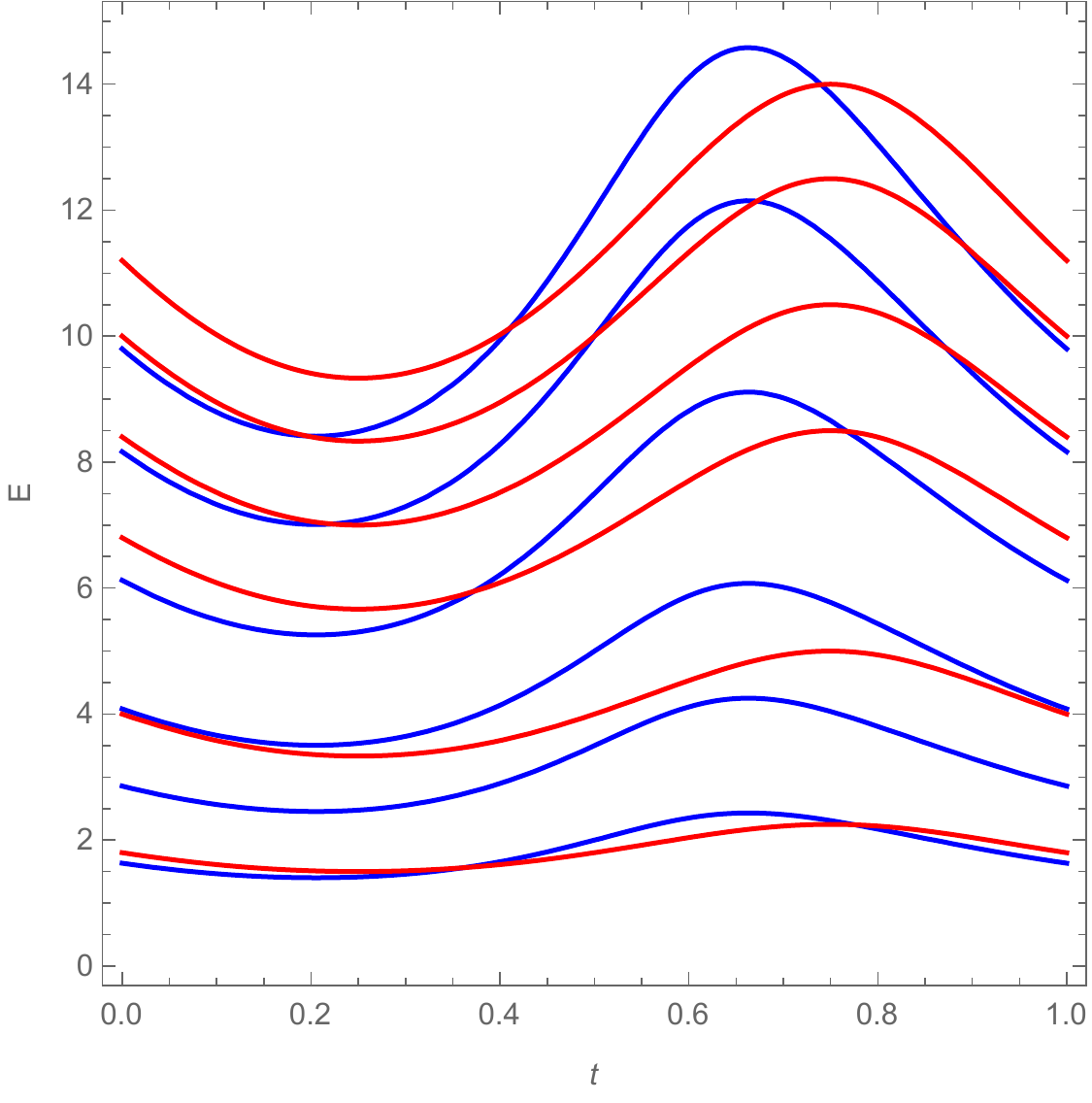}
    \caption{The level lines of $H_{in}$ (in red) and $H_{out}$ (in blue), in scaled variables $(E,t)$, for $a(t) = 5 + \sin (2\pi t)$, $b(t) = 2 - \cos (2 \pi t)$.}
    \label{fig:}
\end{figure}

Namely, we find that the strong transversality is only satisfied in the limit of large energy if $\delta> 0$, and 
that at $\delta=0$ the projection of the primary homoclinic set to the cylinder $\Lambda$ by the strong unstable 
fibers shrinks to a set of parallel lines $t=const$ as $\varepsilon=0$. To this aim, we put $\delta=0$ 
and investigate the splitting of invariant manifolds  $W^{s,u}(\Lambda)$. The first term of the power expansion 
in $\varepsilon$ for the distance between $(\theta^{s}_{\varepsilon, 0} - \theta^{u}_{\varepsilon,0})$ between
perturbed $W^{s,u}(\Lambda)$ is given
by the so-called Melnikov function $M_{1}(\varphi, \theta, E, t)$. Non-degenerate zeros of $M_1$ correspond
to transverse primary intersections of 
$W^{s}$ and $W^{u}$, if $\varepsilon$ is small enough. We show
\begin{theorem}
Consider the time-dependent elliptic billiard map $B$ \textit{without} the quartic perturbation (i.e. $\delta=0$). 
The Melnikov function associated to the splitting of invariant manifolds  $W^{s,u}(\Lambda)$ has zeroes only
for the times $t^{*}$ such that $\frac{d}{dt}\left(\frac{a(t^{*})}{b(t^{*})}\right) = 0$, for all values of 
energy and reflection angle. If $t^{*}$ is a nondegenerate critical point of $\frac{a(t)}{b(t)}$, there exists
a corresponding transverse intersection of $W^{s,u}(\Lambda)$ along a $2$-dimensional homoclinic surface where 
$(\theta,t)=(\theta^{s,u}_{0,0}, t^*)+O(\epsilon)$ are smooth functions of $(\mathcal{E},\varphi)$.
\end{theorem}

In the cylinder $\Lambda$, the image of the two-dimensional homoclinic intersection found in this theorem by the 
projection along the strong-unstable fibers is confined in a narrow strip around the critical lines $t=t^*$. This means
that the scattering map is not properly defined if $\delta=0$ (to estimate the domain of definition of the scattering 
map we would need further expansion of the separatrix splitting function in powers of $\varepsilon$, but we 
suspect that it is small beyond all orders). We conjecture that the same structure is characteristic of a more general case
of an integrable system with slowly varied parameters.  

In order to have a scattering map defined, we add the $\delta$-dependent term in (\ref{ovalss}).
The nonintegrability of static elliptic billiards subject to polynomial perturbations was studied 
in \cite{delshams1996poincare, lomeli1996perturbations, tabanov1994separatrices, levallois1993separation}. 
While we use the Melnikov function calculations from these works, we also develop a novel Melnikov function technique
for the computation of the scattering map for systems with normally-hyperbolic invariant manifolds 
(e.g. time-dependent billiards). In particular, we show that the domain $\bar \Lambda$
of definition of the scattering map $S_{\Gamma}$ 
in our situation has an unusual shape at small $\delta$ - it contains essential curves only at very large energies. 
Namely, to the first order in $\frac{1}{\sqrt{\mathcal{E}}}$ and $\delta$ the domain $\bar{\Lambda}$ 
is given by 
$$\sqrt{\mathcal{E}} > \frac{|a\frac{db}{dt} - b \frac{da}{dt}|}{| \delta|}  \;\phi(t),$$
where $\phi(t)$, defined by (\ref{phi}) is a strictly positive, continuous, periodic function of $t$. 
More precisely, we have the following

\begin{theorem}
For any constant $k>0$ there exists $C>0$ such that 
all points $(\mathcal{E},t)$ in the cylinder $\Lambda$, which satisfy
\begin{equation}
\sqrt{\mathcal{E}} > \frac{|a\frac{db}{dt} - b \frac{da}{dt}|}{| \delta|}  \phi(t) + k, \qquad
\mathcal{E} \geq \frac{C(k)}{\delta},
\label{eqbarlambda+}
\end{equation}
belong to the domain $\bar{\Lambda}$ of the definition of the scattering map $S_{\Gamma}$.
Vice versa, the points which satisfy
\begin{equation}
\sqrt{\mathcal{E}} < \frac{|a\frac{db}{dt} - b \frac{da}{dt}|}{| \delta|}  \phi(t) - k,
\label{eqbarlambda-}
\end{equation}
do not belong to $\bar{\Lambda} \cap \{\mathcal{E} \geq \frac{C(k)}{\delta}\}$.
\end{theorem}
It is seen from these formulas that $\bar{\Lambda}$ contains a circle $\mathcal{E}=const$ only if 
$\mathcal{E}> C_{1} \delta^{-2}$ where $C_{1}$ is some constant. In this region of energies techniques of
\cite{gelfreich2008fermi} can be applied in order to prove the existence of orbits for which the energy tends to 
infinity starting from
$\mathcal{E}\sim \delta^{-2}$. Our Theorem 1.1 gives a stronger result by allowing to start at much lower energies 
$\mathcal{E}\sim \delta^{-1}$. 

\begin{figure}[h]
    \centering
    \includegraphics[width=0.6\textwidth]{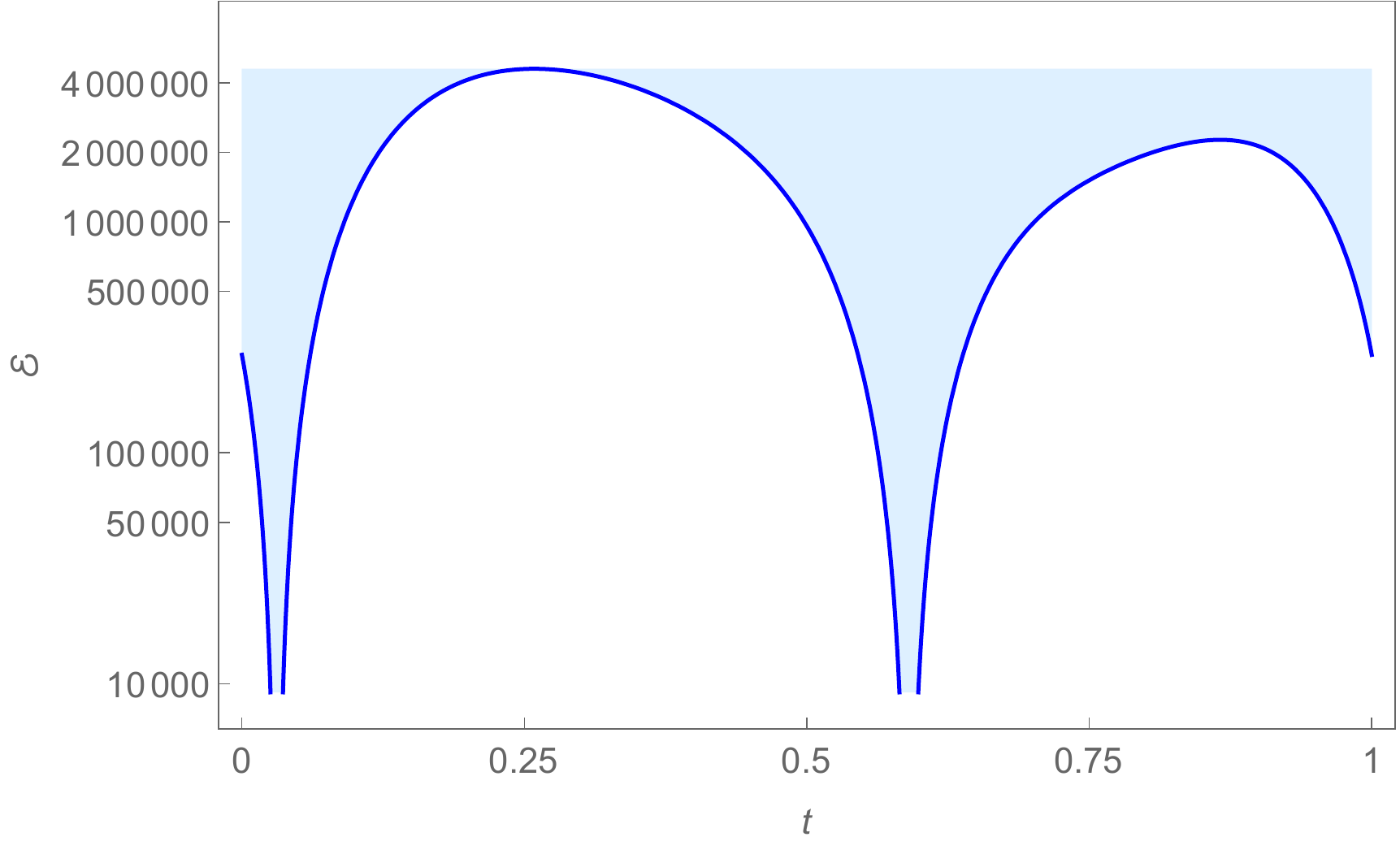}
    \caption{The domain of definition $\bar{\Lambda}$ of the scattering map $S_{\Gamma}$, in physical variables $(\mathcal{E}, t)$. The dark blue curve is the graph of 
$ \left(\frac{|a\frac{db}{dt} - b \frac{da}{dt}|}{| \delta|}  \;\phi(t)\right)^{2}$, with $\phi(t)$ given by (\ref{phi}). The shaded region is the domain $\bar{\Lambda}$. Here $a(t) = 5 + \sin (2\pi t)$, $b(t) = 2 - \cos (2 \pi t)$, $\delta = 0.05$.}
    \label{fig:}
\end{figure}

Our paper is organised as follows. In section 2 we review the known facts about the static elliptic billiard. Section 3 introduces the time-dependent, perturbed billiard map, where we show how the rescaling of billiard speed gives rise to a slow-fast billiard map. In Section 4 we study the inner and scattering maps. We compute the splitting of stable and unstable invariant manifolds of $\Lambda$ and use this result to derive a first order formula for the scattering map, and therefore give proofs of Theorems 1.2 and 1.3. We also derive the Hamiltonians $H_{in}$ and $H_{out}$ that give the first order approximations to inner and scattering maps. We provide the estimates on energy growth via asymptotic study of inner and outer Hamiltonian vector fields and provide a proof of Theorem 1.1 in Section 5.   The Appendices A, B, C, D and E contain the derivation of the Melnikov function giving the first order distance between perturbed invariant manifolds $W^{s,u}(\Lambda)$ and its explicit computation with elliptic functions; they also provide computations for the scattering map.

\section{Static elliptic billiard}
The following facts are well-known, see for example \cite{tabanov1994separatrices, delshams1996poincare, delshams2001homoclinic}. Our exposition follows \cite{tabanov1994separatrices}. Let us consider a billiard inside a static ellipse.  In Cartesian coordinates, we may define the analytic boundary $Q$ of the ellipse by

\begin{equation}
 Q = \lbrace  (x,y) \in \mathbb{R}^{2}: \frac{x^{2}}{a^{2}} + \frac{y^{2}}{b^{2}} = 1\rbrace,
\label{ellipsecartesian}
\end{equation}
\\
where $0<b<a$. Here $a$ is the half-length of the major axis and $b$ is half-length of the minor axis. The foci are at $(\pm c,0)$ where $c = \sqrt{a^{2}-b^{2}}$. Let us parameterise the ellipse as $\gamma(\varphi): [0,2\pi) \rightarrow Q$,
where
\begin{equation}
\gamma (\varphi) = \{(a\cos (\varphi), b\sin (\varphi)): \varphi \in [0,2\pi)\}.
\label{ellipse}
\end{equation}
\\
Let us introduce the angle of reflection $\theta \in (0,\pi)$ of the particle velocity vector made with the positive tangent to $\gamma(\varphi)$ at the collision point.  We  define the static billiard map $B_{s}:(\varphi_{n}, \theta_{n}) \mapsto (\varphi_{n+1}, \theta_{n+1})$ (with subscript $s$ for `static'). Observe that $Q$ is symmetric with regard to the origin. As in the work by Tabanov \cite{tabanov1994separatrices}, we may exploit this symmetry for $B_{s}$, by identifying the points on the ellipse that are $\pi$ across, hence defining $\varphi \Mod{\pi}$. The following formulas for $B_{s}$ are known \cite{tabanov1994separatrices}:
\begin{equation}
\varphi_{n+1} = -\varphi_{n} + 2 \arctan\left(\frac{b(a \tan(\varphi_{n})+ b \tan (\theta_{n}))}{a(b -a \tan (\varphi_{n}) \tan(\theta_{n}))}\right) \Mod{\pi},
\label{stuff}
\end{equation}
\begin{equation}
\theta_{n+1} = -\theta_{n} + \arctan\left(\frac{b}{a\tan(\varphi_{n})}\right) - \arctan\left(\frac{b}{a\tan(\varphi_{n+1})}\right) \Mod{\pi}.
\label{stuff2}
\end{equation}
\\
The map $B_{s}$ is analytic and preserves the symplectic form $|\gamma'(\varphi)| d\varphi \wedge d\theta$, that becomes standard symplectic form $ds \wedge d(\cos (\theta))$ in coordinates $(s, \cos (\theta))$, where $s$ is the arc length associated to $Q$. The map $B_{s}$ has a hyperbolic saddle fixed point $z=(0,\pi/2)$ with eigenvalues $\lbrace \lambda, \frac{1}{\lambda} \rbrace$, where
\begin{equation}
\lambda = \frac{a+c}{a-c} > 1.
\label{evalue}
\end{equation}
The other fixed point $(\pi/2, \pi/2)$ is elliptic.

The elliptic billiard is integrable: the first integral $I$ of $B_{s}$ may be physically interpreted as the conservation of the inner product of angular momenta about the foci. The integral $I$ may be written as
\begin{equation}
I(\varphi, \theta) = b^{2} \cos^{2}(\theta) - c^{2}\sin^{2} (\theta) \sin^{2} (\varphi).
\label{integs}
\end{equation}
\\

\begin{remark} Tabanov \cite{tabanov1994separatrices} gives the integral as $\tilde I = \cosh^{2} \mu \cos^{2} (\theta) + \cos^{2} (\varphi) \sin^{2} (\theta)$ in the elliptical coordinates $x = h \cosh \mu \cos (\varphi), \quad y = h \sinh \mu \sin (\varphi)$, where $h^{2} = a^{2}-b^{2}$. Upon changing from elliptical coordinates to the parameterisation $\gamma (\varphi)$ above and using $a^{2}-b^{2}=c^{2}$, we have
$$\tilde I=\frac{a^{2}}{c^{2}} \cos^{2} (\theta) + \cos^{2}(\varphi) \sin^{2} (\theta),$$
which gives us
$$c^{2}\tilde I(\varphi, \theta) = a^{2} \cos^{2} (\theta) + c^{2} \cos^{2} (\varphi) \sin^{2} (\theta) = b^{2} \cos^{2} (\theta) - c^{2}\sin^{2} (\varphi) \sin^{2} (\theta) +c^{2}.$$
Rearranging, we have
$$I(\varphi,\theta) = c^{2}\tilde I - c^{2} = b^{2} \cos^{2} (\theta) - c^{2}\sin^{2} (\varphi) \sin^{2} (\theta),$$
which gives us (\ref{integs}). \end{remark}

The level set $I=-c^{2}$ corresponds to the elliptic fixed point; for $-c^{2}<I<0$ the billiard trajectories cross the major axis between the foci and have confocal hyperbolas as caustics, and for $0<I<b^{2}$ trajectories cross the major axis outside the foci and have confocal ellipses as caustics. Zero level set, $I(\varphi,\theta)=0$, corresponds to the union of homoclinic orbits that comprise two coincident branches $W_{1,2}=W^{s}_{1,2}(z)=W^u_{1,2}(z)$, the  stable and unstable  manifolds of $z$. From $I(\varphi, \theta)=0$, the union $W(z)= W_{1}(z) \bigcup W_{2}(z)$ is given by the expression \cite{tabanov1994separatrices}:
\begin{equation}
\sin^{2}(\varphi) = \frac{b^{2}}{c^{2}\tan^{2}(\theta)}.
\label{stable}
\end{equation}
\\

Physically, $W_{1}(z)$ corresponds to the billiard trajectory segments repeatedly passing through the focus at $(c,0)$ while $W_{2}(z)$ correspond to trajectories passing the focus at $(-c,0)$. These trajectories asymptotically tend to the major axis of the ellipse, which corresponds to the saddle fixed point $z$ of $B_{s}$  (recall that we take $\varphi$ by modulo $\pi$).

One can obtain explicit expressions for $B_{s}^{n} \vert_{W_{1,2}(z)}$ for $n  \in \mathbb{Z}$. We have
\begin{equation}
B_{s}^{n}(\varphi_{0},\theta_{0})|_{W_{1}(z)} = \left(2 \arctan\left(\lambda^{n} \tan \left( \frac{\varphi_{0}}{2}\right) \right),\arctan \left(\frac{-b}{c\sin (\varphi_{n})}\right) \right),
\label{w1}
\end{equation}
\begin{equation}
B_{s}^{n}(\varphi_{0},\theta_{0})|_{W_{2}(z)} = \left(2 \arctan\left(\lambda^{-n} \tan \left( \frac{\varphi_{0}}{2}\right)\right),\arctan \left(\frac{b}{c\sin (\varphi_{n})}\right) \right).
\label{w2}
\end{equation}

Let us introduce the variable $\xi \in (0, \infty)$ such that $\xi_{n}=\tan\left(\varphi_{n}/  \right)2$, $n \in \mathbb{Z}$, so that (\ref{w1}) gives
\begin{equation}
\xi_{n} = \lambda^{n}\xi_{0}, \qquad \tan (\theta_{n}) = \frac{-b(1+\xi^{2}_{n})}{2c \xi_{n}},
\label{param}
\end{equation}
while (\ref{w2}) gives
\begin{equation}
\xi_{n}=\lambda^{- n}\xi_{0}, \qquad \tan (\theta_{n}) = \frac{b(1+\xi^{2}_{n})}{2c\xi_{n}}.
\label{parameterisation}
\end{equation}

\begin{remark} Upon making the change of coordinates  $\nu = \|\gamma'(\varphi) \| \cos (\theta)$, as in \cite{delshams1996poincare}, the expression (\ref{stable}) becomes $\nu = \pm c \sin (\varphi)$ and the phase portrait of $B_{s}$  resembles the one of the pendulum Hamiltonian $H = p^{2}/2 + \cos (q) - 1$. In spite of the integrability, the existence of the hyperbolic fixed point with a homoclinic orbit implies that global action-angle variables cannot be introduced in an elliptic billiard: it is an example of an \textit{apriori unstable} system. \end{remark}

\section{Time-dependent perturbed elliptic billiard}

\subsection{Billiard map setup}

Let us consider a billiard inside a time-dependent convex curve  $Q(q,t, \delta)$ that is $O(\delta)$ quartic polynomial perturbation of the ellipse for each fixed time $t$:
\begin{equation}
Q(q,t; \delta) := \lbrace q = (x,y) \in \mathbb{R}^{2}, \quad t \in \mathbb{S}^{1}: \frac{x^{2}}{a^{2}(t)} + \frac{y^{2}}{b^{2}(t)} = 1 + \frac{2 \delta y^{4}}{b^{4}(t)} \quad \rbrace,
\label{timebound}
\end{equation}
where $a$ and $b$ are periodic $C^{r+1}$ ($r \geq 4$)  functions  of time $t$, such that $0<b(t)<a(t)$ for all $t$, and $0 <| \delta | \ll 1$. Let us parameterise $Q(q,t; \delta)$ as
\begin{equation}
Q(q,t; \delta) = \{(a(t)\cos (\varphi), b(t)\sin (\varphi) \left(1+ \delta \sin^{2} (\varphi) \right) ) + O(\delta^{2}),  \varphi \in [0, 2 \pi), t \in \mathbb{S}^{1}\}.
\label{ellipse}
\end{equation}
The $O(\delta^{2})$ terms do not play any role in our work.

\begin{remark} Polynomial perturbations of billiards in ellipses were considered previously in a number of works \cite{delshams1996poincare, lomeli1996perturbations, tabanov1994separatrices, levallois1993separation}, however these works considered static perturbations only, not time-dependent ones.
\end{remark}

Assuming the billiard reflection at the moment of collision with the moving boundary is elastic, we define the time-dependent billiard map $B: (\varphi_{n}, \theta_{n}, \mathcal{E}_{n}, t_{n}) \mapsto (\varphi_{n+1}, \theta_{n+1}, \mathcal{E}_{n+1}, t_{n+1})$. Here $\varphi_{n}$ is the collision point on the boundary at the $n$-th collision, $\theta_{n}$ is the reflection angle of the post-collision particle velocity vector made with the positively oriented tangent to the boundary at the collision point, $\mathcal{E}_{n}$ is the particle post-collision energy, and $t_{n}$ is the time of the $n$-th collision. It is known that $B$ is symplectic hence volume-preserving \cite{koiller1995time, gelfreich2012fermi}. Since the boundary curve (\ref{ellipse}) is analytic with respect to $\varphi$ and $C^{r+1}$ in $t$, it is known that $B$ is a $C^{r}$ diffeomorphism \cite{koiller1995time}.

Denote the speed of the particle as $w$, its corresponding velocity as $\mathbf{w}$, its energy as $\mathcal{E}= \frac{w^{2}}{2}$; the speed of the boundary in the direction of outward normal is given by $u(q,t)= -\frac{1 }{\nabla_{q} Q(q,t)} \frac{\partial Q(q,t; \delta)}{\partial t}$, and the unit outward normal is $\mathbf{n} = \frac{\nabla_{q} Q(q,t; \delta)}{\| \nabla_{q} Q(q,t; \delta)\|}$. We assume that positive $u$ corresponds to outward motion of the boundary. The following formula \cite{gelfreich2012fermi} gives the change in velocity at the boundary collision:
\begin{equation}
\mathbf{w}_{n+1} = \mathbf{w}_{n}-2\langle \mathbf{w}_{n},\mathbf{n}_{n+1} \rangle\ \mathbf{n}_{n+1} +2u(\varphi_{n+1}, t_{n+1}) \mathbf{n}_{n+1}.
\label{changemom}
\end{equation}

By analogy with \cite{koiller1995time}, let us introduce the auxilliary variable $\theta^{*}$, which denotes the angle of incidence at the $(n+1)$-th impact with the tangent to the boundary, and let $\alpha$ denote the angle between the tangent to $Q(q,t; \delta)$ and the $x$-axis, given by $\tan (\alpha) = \frac{y'(\varphi)}{x'(\varphi)}$, with $' = \frac{d}{d\varphi}$ (at each fixed $t$) and $x(\varphi)$, $y(\varphi)$ defined by (\ref{ellipse}). Since (\ref{ellipse}) is symmetric with respect to the origin, let us define $\varphi \Mod{\pi}$ as in section 2, thus identifying points on the boundary that are $\pi$ across.  Define $u( \varphi_{n+1}, t_{n+1})$ to be the normal speed of the boundary at $(n+1)$-th impact. In this way,  we obtain an implicit form for the billiard map
$$B: (\varphi_{n}, \theta_{n}, \mathcal{E}_{n}, t_{ n}) \mapsto (\varphi_{n+1}, \theta_{n+1}, \mathcal{E}_{n+1}, t_{n+1}),$$ 
given by the following formulas (more details can found in \cite{koiller1995time}):
\begin{align}
\begin{split}
  a(t_{n+1}) \cos (\varphi_{n+1}) & = a(t_{n}) \cos (\varphi_{n}) + \sqrt{2\mathcal{E}_{n}}(t_{n+1} - t_{n}) \cos( \alpha_{n}+ \theta{_n}),
\\
b(t_{n+1})\sin (\varphi_{n+1}) \left(1+ \delta \sin^{2} (\varphi_{n+1}) \right) & = b(t_{n})\sin (\varphi_{n}) \left(1+ \delta \sin^{2} (\varphi_{n}) \right) + \sqrt{2\mathcal{E}_{n}}(t_{n+1}-t_{n}) \sin (\alpha_{n} + \theta_{n}),
 \\
 \theta_{n} + \alpha_{n} + \theta^{*}_{n+1} - \alpha_{n+1} & = 0,
 \\
 \sqrt{2\mathcal{E}_{n+1}}\cos (\theta_{n+1}) & = \sqrt{2\mathcal{E}_{n}}\cos(\theta^{*}_{n+1}),
 \\
 \sqrt{2\mathcal{E}_{n+1}}\sin(\theta_{n+1}) & = \sqrt{2\mathcal{E}_{n}}\sin(\theta^{*}_{n+1})-2u(\varphi_{n+1}, t_{n+1}).
\end{split}
\label{system}
\end{align}

The first pair of equations of (\ref{system}) implicitly defines $t_{n+1}$ and $\varphi_{n+1}$, while the last three give $\mathcal{E}_{n+1}$ and $\theta_{n+1}$ after expressing $\theta^{*}_{n+1}$ in terms of $\varphi_{n+1}$, $\theta_{n}$ and $\varphi_{n}$. The last pair of equations in (\ref{system}) corresponds to (\ref{changemom}) written in components normal and tangential to the boundary, and they give the expression for the change of energy
\begin{equation}
\mathcal{E}_{n+1} = \mathcal{E}_{n} - 2 \sqrt{2E_{n}}u(\varphi_{n+1}, t_{n+1}) \sin (\theta^{*}_{n+1}) + 2u^{2}(\varphi_{n+1}, t_{n+1}).
\label{speedchange}
\end{equation}

Let us denote by $D$ the Euclidean distance between $\varphi_{n}$ and $\varphi_{n+1}$, then we have
\begin{equation}
t_{n+1} = t_{n}+\frac{D}{\sqrt{2\mathcal{E}_{n}}},
\label{flightime}
\end{equation} 
where the expression 
\begin{equation}
D = \sqrt{\left[a(t_{n+1}) \cos (\varphi_{n+1}) - a(t_{n})\cos (\varphi_{n})\right]^{2} + \left[b(t_{n+1})\sin (\varphi_{n+1}) \left(1+ \delta \sin^{2} (\varphi_{n+1}) \right)) - b(t_{n})\sin (\varphi_{n}) \left(1+ \delta \sin^{2} (\varphi_{n}) \right)\right]^{2}}
\label{euclid}
\end{equation}
is obtained from the first two equations of (\ref{system}). 

We assume that the initial speed of the particle is much larger than the speed of the boundary, so that the shape of the billiard table does not change significantly from one impact to the next. This implies that the time interval between two consecutive collisions is small and the change in the speed of the particle due to a single collision is small compared to the initial speed 
of the particle. Motivated by this, let us write the billiard map in a ``slow-fast" form. Let us take an initial large value of speed $w^{*}$ and introduce a small parameter $\varepsilon = \frac{1}{w^{*}}$ such that $0 < \varepsilon \ll 1$.  Let us introduce the scaled speed $v$ that is related to the original physical variable $w$ through $v = \frac{w}{w{*}}$. In terms of $\varepsilon$ this gives
\begin{equation}
w =  \frac{v}{\varepsilon}.
\label{scale}
\end{equation}

This transformation is equivalent to $\mathcal{E}  = \frac{E}{\varepsilon^{2}}$, where $E$ is the \textit{rescaled} energy. The billiard map in the rescaled energy and time variables becomes close to identity, since  (\ref{speedchange}) transforms to
\begin{equation}
E_{n+1} = E_{n} - 2 \varepsilon \sqrt{2E_{n}}u(\varphi_{n+1}, t_{n+1}) \sin (\theta^{*}_{n+1}) + 2 \varepsilon^{2} u^{2}(\varphi_{n+1}, t_{n+1}),
\label{trans1}
\end{equation}
or, in terms of $v$,
\begin{equation}
v_{n+1} = v_{n} - 2\varepsilon u(\varphi_{n+1}, t_{n+1}) \sin (\theta^{*}_{n+1}) + O(\varepsilon^{2})
\end{equation}
Using (\ref{scale}) transforms the equation (\ref{flightime}) to
\begin{equation}
t_{n+1} = t_{n}+ \frac{\varepsilon D}{\sqrt{2E_{n}}}.
\label{time}
\end{equation}
Note that in the limit $\varepsilon\to 0$ the variables $(t,E)$ become constants, i.e., the billiard map coincides with the frozen billiard map in the domain bounded by (\ref{timebound}).

Now, by virtue of $C^{r}$-smoothness of the boundary and smallness of $\varepsilon$ and $\delta$  we write the time-dependent 
billiard map $B = B_{\varepsilon, \delta} (\varphi_{n}, \theta_{n}, E_{n}, t_{n}) \mapsto (\varphi_{n+1,\varepsilon, \delta}, \theta_{n+1,\varepsilon, \delta}, E_{n+1,\varepsilon, \delta}, t_{n+1,\varepsilon, \delta})$ in the form $B_{\varepsilon, \delta}  = B_{0} + \varepsilon B_{1} + \delta B_{2} + O(\varepsilon^{2} + \delta^{2})$. We define

\begin{align}
\begin{split}
\varphi_{n+1,\varepsilon, \delta}=\varphi_{n+1}+ \varepsilon f_{1}(\varphi_{n},\theta_{n}, E_{n},t_{n}) + \delta g_{1}(\varphi_{n},\theta_{n}, t_{n}) + O(\varepsilon^{2} + \delta^{2}),
\\
\theta_{n+1,\varepsilon, \delta}=\theta_{n+1}+ \varepsilon f_{2}(\varphi_{n},\theta_{n}, E_{n}, t_{n}) + \delta g_{2}(\varphi_{n}, \theta_{n}, t_{n}) + O(\varepsilon^{2} +  \delta^{2}),
\\
E_{n+1,\varepsilon}=E_{n} + \varepsilon f_{3}(\varphi_{n},\theta_{n},E_{n},t_{n})  + \varepsilon O(\varepsilon +  |\delta|),
\\
t_{n+1,\varepsilon}=t_{n} + \varepsilon f_{4}(\varphi_{n},\theta_{n},E_{n},t_{n})  + \varepsilon O(\varepsilon +  |\delta|).
\end{split}
\label{components}
\end{align}

Here $B_{1} = (f_{1},f_{2},f_{3},f_{4})^{\top}$ and $B_{2} = (g_{1},g_{2},0,0)^{\top}$ (with ${}^\top$ denoting the transpose). We use the notation $(\varphi_{n+1}, \theta_{n+1}) = B_{s}(\varphi_{n}, \theta_{n})$, i.e. $B_{0}$ is the same as the static two-dimensional billiard map $B_{s}$ in the ellipse (\ref{ellipsecartesian}), with the increase of the phase space dimension to account for the two "frozen" variables $E$ and $t$; so, $B_{0} ( \varphi, \theta, E, t) = (B_{s}(\varphi, \theta), E,t)$. 

Let us call $B_{0}$ the \textit{unperturbed} time-dependent elliptic billiard map. The map $B_{0}$ is integrable with the first integral (\ref{integs}) where $b = b(t)$ and $c=c(t)$ are fixed, and two more trivial first integrals $I_{2}=E$, $I_{3}=t$. The phase space of $B$ is $[0, \pi) \times (0, \pi) \times \mathbb{R}^{+} \times \mathbb{S}^{1}$, with $\varphi \in [0,\pi)$, $\theta \in (0, \pi)$, $E \in \mathbb{R}^{+}$, and $t \in \mathbb{S}^{1}$. 

\begin{remark}
Since we consider high billiard energies, we assume that the billiard reflection angle $\theta \in (0, \pi)$, i.e. the situations described in \cite{koiller1995time} where the billiard trajectory continues in a tangential or `outward' direction to the boundary at the moment of collision do not occur.
\end{remark} 

By substituting the expression for $t_{n+1, \varepsilon, \delta}$ and $\varphi_{n+1, \varepsilon, \delta}$ from (\ref{components}) into (\ref{euclid}) and expanding in Taylor series, we find the zero-th order in $\varepsilon$ and $\delta$ free-flight distance
\begin{equation}
D_{0} = \sqrt{a^{2}[\cos(\varphi_{n}) - \cos(\varphi_{n+1})]^{2} + b^{2}[\sin(\varphi_{n}) - \sin(\varphi_{n+1})]^2}.
\label{euclidzero}
\end{equation}

Upon substituting expansion (\ref{components}) into (\ref{system}), and examining the coefficients of the order $\varepsilon$ terms, we find that $B_{1} = (f_{1},f_{2},f_{3},f_{4})^{\top}$ is given by the following expressions:
\begin{equation}
f_{1}  = \frac{D_{0}}{\sqrt{2E_{n}}}\left(\frac{\dot{a}\cos(\varphi_{n+1}) \tan(\theta_{n}+\alpha_{n}) - \dot{b}\sin(\varphi_{n+1})}{a\sin(\varphi_{n+1}) \tan(\theta_{n}+\alpha_{n}) + b\cos(\varphi_{n+1})}\right),
\label{f1}
\end{equation}
\begin{equation}
f_{2} = \frac{-2u \cos (\theta_{n+1})}{\sqrt{2E_{n}}} + \frac{a^{2} \sin^{2} (\varphi_{n+1})}{a^{2} \sin^{2} (\varphi_{n+1}) + b^{2} \cos^{2} (\varphi_{n+1}) } \left(\frac{D_{0}(\dot{a}a^{-1}b - \dot{b})}{a \sqrt{2E_{n}} \tan (\varphi_{n+1})} + \frac{b f_{1}}{a \sin^{2} (\varphi_{n+1})}\right),
\label{f2}
\end{equation}
\begin{equation}
f_{3}= -2\sqrt{2E_{n}} u \sin (\theta_{n+1}),
\label{f3}
\end{equation}
\begin{equation}
f_{4} = \frac{D_{0}}{\sqrt{2E_{n}}}.
\label{f4}
\end{equation}

The dot above  $a$ and $b$ denotes the derivative with respect to time evaluated at time $t_{n}$. We also denote $a = a(t_{n}), b=b(t_{n})$, and $u =  \frac{\dot{a}b\cos^{2}(\varphi_{n+1}) + a\dot{b}\sin^{2}(\varphi_{n+1})}{\sqrt{ a^{2}\sin^{2}(\varphi_{n+1}) + b^{2}\cos^{2}(\varphi_{n+1})}}$ is the normal speed of the boundary. Similarly, comparing the coefficients of the first order in $\delta$, we find that the expressions for $g_{i}$ for $i=1,..,4$ are independent of $E$ 
and $\dot{a}$, $\dot{b}$. 

\begin{remark} Observe that $f_{1}, f_{2}, f_{3}, f_{4}$ are written in a certain ``cross-form" as functions of the image of $(\varphi, \theta, E, t)$ under $B_{0}$ as well as the initial point $(\varphi, \theta, E, t)$ itself; however since $B$ is a $C^{r}$ diffeomorphism, one may express the functions $f_{i}$ in form (\ref{components}). \end{remark}

\begin{remark} Observe that the Taylor series expansion of the map $B_{\varepsilon, \delta}$ consists of two perturbations that may be considered independently in the first order of the perturbation parameters: the $O(\varepsilon)$ perturbation terms that arise due to rescaling of energy (this is $B_{1}$), and the $O(\delta)$ perturbation terms that correspond to the polynomial perturbation of the boundary (this is $B_{2}$, which is independent of $t$).  \end{remark}

\begin{remark} For $\varepsilon = 0$, the variables $E$ and $t$ are constant and thus $B_{0,\delta}$ becomes a  billiard map corresponding to an ellipse with a quartic polynomial perturbation, with the semi-axes lengths fixed at $a(t_{n})$, $b(t_{n})$. This is a twist map and thus has a certain generating function $L(\varphi_{n}, \varphi_{n+1})$ \cite{delshams1996poincare} . The nonintegrability of such convex billiard curves and the relation between generating function, Melnikov function and Melnikov potential was investigated in detail in \cite{delshams1996poincare}. Thus, for the computation of the Melnikov function we do not require the knowledge of the explicit form of $B_{2}$, as we will be using the generating function formulation and Melnikov potential \cite{delshams1996poincare} for the polynomial part of the perturbation. \end{remark}

\subsection{Phase space geometry of the time-dependent billiard}

If a point in our billiard moves along the major semi-axis, it will continue to move along this semi-axis forever. In other
words, this motion is confined to an invariant manifold $\Lambda$ in the phase space of the billiard map $B$. It is given by
$(\varphi \Mod{\pi},\theta)=(0,\pi/2)$ and is parameterized by the energy $\mathcal{E}\in \mathbb{R}^{+}$ and the time 
$t\in\mathbb{S}^1$, so $\Lambda$ is a cylinder. If we rescale the energy and take the limit $\varepsilon=0$, the cylinder
$\Lambda$ corresponds to the hyperbolic saddle fixed point $z=(0,\pi/2)$ of the static billiard map $B_{s}$, so $\Lambda$
is a normally-hyperbolic invariant manifold of $B_{0,\delta}$.

It has two branches of stable and unstable three-dimensional invariant manifolds $W^{s,u}_{i}(\Lambda)$ with $i = 1,2$.
At $\delta=0$, they coincide and form two symmetric three-dimensional homoclinic manifolds $W_{i}$ that are inherited from one-dimensional separatrices of $z$ (see (\ref{stable})). Denoting $W = \cup_{i=1}^{2} W_{i}$, we have
\begin{equation}
W = \{(\varphi, \theta, E,t):\varphi \in [0,\pi), \theta \in (0,\pi), E \in \mathbb{R}^{+}, t \in \mathbb{S}^{1}; \quad \sin^{2}(\varphi) = \frac{b^{2}(t)}{c^{2}(t)\tan^{2}(\theta)} \ \}.
\label{invmanifoldsfull2}
\end{equation}

By the theory of Fenichel \cite{fenichel1971persistence}, the normal hyperbolicity of $\Lambda$ persists as the map $B_{0,\delta}$ is perturbed. In particular, the stable and unstable manifolds of any subset of $\Lambda$ that corresponds to a bounded set of values of the rescaled energy $E$ persist at all small $\varepsilon$, and depend continuously on $\varepsilon$ (and $\delta$). If we return to the non-rescaled energy variable $\mathcal{E}$, this gives us, for all small $\delta$, the normal hyperbolicity, for the time-dependent billiard, of the piece of $\Lambda$ that corresponds to sufficiently large values of $\mathcal{E}$; with the stable and unstable manifolds $W^{s,u}(\Lambda)$ close to those for the frozen billiard map and
depending continuously on $\delta$.

At $\varepsilon=0$, the stable and unstable manifolds $W^{s,u}_{i}$ are foliated by strong stable and strong unstable one-dimensional fibers $W^{ss, uu}_{(E,t);i}: (E,t)=const$. They form smooth invariant foliations transverse to $\Lambda$; such foliations are unique and persist at small smooth perturbations of the system \cite{fenichel1971persistence}. Thus, these invariant foliations persist for the time-dependent billiard as well, and depend continuously on $\delta$ and, when the rescaling of the energy variable is introduced, on $\varepsilon$.

The closeness of $B$ to identity (in the rescaled energy $E$ and time $t$) implies a large spectral gap \cite{gelfreich2014arnold} for the normally-hyperbolic cylinder $\Lambda$. Thus, the stable and unstable manifolds $W^{s,u}(\Lambda)$ are $C^{r}$ and their leaves $W^{ss,uu}$ are $C^{r-1}$ in $(\varphi, \theta,E,t)$ and also in the parameters $(\varepsilon, \delta)$, for $r \geq 4$ \cite{fenichel1971persistence}.

Note also, that for the unperturbed system (i.e., at $\varepsilon=0$ and $\delta=0$) the fibers $W^{ss}_{(E,t);i}$ and $W^{uu}_{(E,t);i}$ coincide for each $(E,t)$. 

\section{Inner and outer dynamics}
In this section we will define and study the inner and scattering (outer)  maps associated to $\Lambda$. Iterations of these maps will be the main tool for constructing a billiard orbit with growing energy in section 5.
\subsection{Inner map}

The inner map is the restriction of $B$ to $\Lambda$. Physically, the inner map describes the billiard motion along the major axis of the billiard domain. Let us denote the inner map by $\Phi$ and hence $\Phi = B|_\Lambda$ writes as $B(0, \frac{\pi}{2}, E_{n}, t_{n}) = (0, \frac{\pi}{2}, E_{n+1}, t_{n+1})$ (recall that $E$ is the rescaled energy). Therefore, using (\ref{trans1}) and (\ref{time}), the inner map can be given in an implicit form
\begin{align}
\begin{split}
E_{n+1} =E_{n} - 2 \varepsilon \sqrt{2E_{n}} \dot{a}(t_{n+1}) + 2 \varepsilon^{2} \dot{a}^{2}(t_{n+1}),
\\
t_{n+1} = t_{n} + \frac{\varepsilon( a(t_{n}) + a(t_{n+1}))}{\sqrt{2E_{n}}},
\end{split}
\label{innerscaled}
\end{align}

The map $\Phi$ which defines a $C^{r}$ diffeomorphism $(E_{n}, t_{n}) \mapsto (E_{n+1}, t_{n+1})$ for small $\varepsilon$, preserves the symplectic form $\left(1+ \varepsilon \frac{\dot{a}(t)}{\sqrt{2E}}\right) dE \wedge dt$ that becomes standard form $dF \wedge dt$ upon defining $F = \sqrt{E}\left(\sqrt{E} + \varepsilon \frac{2\dot{a}(t)}{\sqrt{2}}\right)$. The symplecticity of the inner map and its closeness to identity imply that it may be approximated by the time-$\varepsilon$ shift along a level curve of an autonomous Hamiltonian $H_{in}(t, E ; \varepsilon)= H_{in}(t,E) + O(\varepsilon)$ defined on $\Lambda$. Let us find $H_{in}$.  A series expansion in $\varepsilon$  yields the following approximation of (\ref{innerscaled}):
\begin{equation}
E_{n+1}=E_{n} - 2\varepsilon \dot{a}(t_{n}) \sqrt{2E_{n}} +O(\varepsilon^{2}) , \qquad t_{n+1} = t_{n} + \frac{\sqrt{2}\varepsilon a(t_{n})}{\sqrt{E_{n}}} + O(\varepsilon^{2}).
\label{innerembed}
\end{equation}
Since (\ref{innerembed}) gives
$$\frac{E_{n+1} - E_{n}}{t_{n+1}-t_{n}} = -\frac{2\dot{a}(t_{n})E_{n}}{a(t_{n})} + O(\varepsilon),$$
we see that $\Phi$ is approximated to $O(\varepsilon^{2})$ by a time-$\epsilon$ shift along a trajectory of the solution of the differential equation 
$$\frac{dE}{dt} = -\frac{2\dot{a}E}{a}.$$

Its integral $\sqrt{E}a(t)=\mbox{const}$ gives the zero-th order approximation of $H_{in}$. Thus,  map (\ref{innerscaled}) up to $O(\varepsilon^{2})$ is given by the time-$\epsilon$ shift along a level curve of the Hamiltonian
\begin{equation}
H_{in}(t,E)= 2\sqrt{2E}a(t).
\label{innh}
\end{equation}
The corresponding Hamiltonian vector field is:
\begin{equation}
\frac{dt}{ds}= \frac{\partial H_{in}}{\partial E} = \frac{\sqrt{2}a(t)}{\sqrt{E}}, \qquad \frac{dE}{ds} = - \frac{\partial H_{in}}{\partial t} = -2\sqrt{2E}\dot{a}(t),
\label{hamiltonseq}
\end{equation}
where $s$ is an auxiliary time variable. 

Let $(\bar{E}, \bar{t})$ be the image of $(E, t )$ under $\Phi^{p}$ where $p = [\frac{1}{\varepsilon}]$. Let us show the following
\begin{proposition}
The inner map $\Phi^{p}: (E, t ) \mapsto (\bar{E}, \bar{t}) $ satisfies the twist condition, i.e., $\frac{\partial \bar{t}}{\partial E} \neq 0$.
\end{proposition}

\begin{proof}
Let us denote by $\phi$ the time-$1$ map of $H_{in}(t,E)$. Observe that $\phi$ is $O(\varepsilon)$ close to $\Phi^{p}$ in $(E,t)$ coordinates. Let us verify that $\phi$ has the twist property.  Since $H_{in}$ is integrable, $\phi$ is also integrable, i.e., it preserves $H_{in}$.  From the first equation of (\ref{hamiltonseq}), we have $\frac{ds}{dt} = \frac{1}{a(t)} \sqrt{\frac{E}{2}}$. Expressing $E$ in terms of $H_{in}$ from (\ref{innh}) on this curve yields $\frac{ds}{dt} = \frac{H_{in}}{4a^{2}(t)}$. Then we have 
\begin{equation}
1 = H_{in} \int_{t}^{\bar{t}} \frac{dt}{4a^{2}(t)} = \sqrt{\frac{E}{2}} a(t) \int_{t}^{\bar{t}} \frac{dt}{a^{2}(t)}
\label{period}
\end{equation}
for the map $\phi$. Differentiating both sides of (\ref{period}) with respect to $E$ yields $\frac{\partial \bar{t}}{\partial E} < 0$. Thus, by definition $\phi$ is a twist map. Since $\Phi^{p}$ is an $O(\varepsilon)$ perturbation of $\phi$, it also has the twist property. 
\end{proof}

Since all orbits of $H_{in}$ are invariant closed curves forming a continuous foliation of $\Lambda$, and $\phi$ is integrable twist map, it follows from the KAM theory that $\Phi^{p}: (E, t ) \mapsto (\bar{E}, \bar{t})$ has closely spaced invariant curves on $\Lambda$. Hence, the energy of the billiard motion with the initial conditions on the major axis is always bounded.

\begin{remark} It follows from standard results \cite{benettin1994hamiltonian} that the inner map (\ref{innerscaled}) coincides up to $O(\varepsilon^{r+1})$ with $\varepsilon$-time shift along a level curve of some Hamiltonian $H_{in}(t,E;\varepsilon)$, which is given by (\ref{innh}) at $\varepsilon = 0$, a fact that we will make use of in section 5.
\end{remark}

\subsection{Scattering map: theory}

The so-called outer dynamics on $\Lambda$ is defined by the scattering map (also called the outer map), studied in detail in \cite{delshams2008geometric}. It is obtained by an asymptotic process: to construct it, one starts infinitesimally close to the normally hyperbolic invariant manifold, moves along its unstable manifold up to a homoclinic intersection and, then, back to $\Lambda$ along its stable manifold.

Let us define the scattering map in a general setup.  Let $T: M \mapsto M$ be a $C^{r}$ (here we take $\quad r \geq 1$)  diffeomorphism on a compact manifold $M$, and let $\Lambda \subset M$ be a compact normally-hyperbolic invariant manifold of $T$. By the normal hyperbolicity, there exits stable $W^{s}(\Lambda)$ and unstable $W^{u}(\Lambda)$ manifolds of $\Lambda$ with strong stable and strong unstable foliations $W^{ss,uu}(x)$ for each $x \in \Lambda$.  Let us assume that  stable and unstable manifolds of $\Lambda$ intersect transversally along a homoclinic manifold  $\Gamma \subset W^{s}(\Lambda) \cap W^{u}(\Lambda)$: for all $z \in \Gamma$ we have
\begin{equation}
T_{z}W^{s}_{\Lambda} + T_{z}W^{u}_{\Lambda} = T_{z}M, \qquad T_{z}W^{s}_{\Lambda} \cap T_{z}W^{u}_{\Lambda} = T_{z} \Gamma.
\label{trans}
\end{equation}

For a given $z \in \Gamma$, there are unique points $x_{\pm} \in \Lambda$ satisfying $z \in W^{ss}_{x_{+}}$ and $z \in W^{uu}_{x_{-}}$.  Following \cite{gelfreich2014arnold} we call the homoclinic intersection at the point $z$ \textit{strongly transverse} if the leaf $W^{ss}_{x_{+}}$ is transverse to $W^{s}(\Lambda)$ and the leaf $W^{uu}_{x_{-}}$ is transverse to $W^{u}(\Lambda)$ at $z$:
\begin{equation}
T_{z}W^{ss}_{x_{+}} \bigoplus T_{z}\Gamma = T_{z}W^{s}_{\Lambda}, \qquad T_{z}W^{uu}_{x_{-}} \bigoplus T_{z}\Gamma = T_{z}W^{u}_{\Lambda}.
\label{transf}
\end{equation}
  
Conditions (\ref{transf}) are used to locally define the {\em scattering map} $S_{\Gamma}$. We say that $\Gamma$ is a homoclinic channel if it satisfies (\ref{trans}) and (\ref{transf}). Let $\pi^{s}: \Gamma \mapsto \Lambda$ and $\pi^{u}: \Gamma  \mapsto \Lambda$ be the projections by the strong stable leaves of stable manifolds, and strong unstable leaves of unstable manifolds, respectively. If we have a sufficiently large spectral gap (the expansion in the strong unstable directions is much larger
than the possible expansion we have in the directions tangent to $\Lambda$), then the strong-stable and strong-unstable
foliations are $C^{r-1}$-smooth \cite{fenichel1971persistence}. This condition is obviously satisfied in our case, since the
restriction of our billiard map to $\Lambda$ is close to identity, i.e. the expansion in the directions tangent to $\Lambda$
can be made as weak as we want. Thus, when conditions (\ref{transf}) are fulfilled, the projections $\pi^{s,u}$ are local $C^{r-1}$ diffeomorphisms. Then, the scattering map $S_{\Gamma}: \Lambda \mapsto \Lambda$ defined as
$$S_{\Gamma} = \pi^{s} \circ (\pi^{u})^{-1}$$
is a $C^{r-1}$ diffeomorphism (which is symplectic if $T$ is symplectic): \cite{delshams2008geometric, gelfreich2014arnold}. 

The invariance property of the strong stable and strong unstable foliations (i.e. the property $T\left(W^{ss,uu}_{x}\right) = W^{ss,uu}_{T(x)}$ for all $x \in \Lambda$) implies that
$$\pi^{s,u} = T^{-1} \circ \pi^{s,u} \circ T, \qquad \pi^{s,u} = T^{-n} \circ \pi^{s,u} \circ T^{n}, \quad n \geq 1. $$
Hence, the dependence of the scattering map on the choice of a particular homoclinic channel obeys the rule
$$S_{\Gamma} = T^{-1} \circ S_{T\left( \Gamma \right)} \circ T.$$

\subsection{Scattering map for $B$ and splitting of invariant manifolds}

Let us take a sufficiently large compact subset $\tilde{\Lambda} \subset \Lambda$ of the normally hyperbolic invariant cylinder $\Lambda$:
\begin{equation}
\tilde{\Lambda} = \{(\varphi, \theta, E, t): (\varphi, \theta) = (0, \pi/2), E \in [E_{1}, E_{2}], t \in \mathbb{S}^{1} \}.
\label{subsetnhim}
\end{equation}
In this section we will find a subset $\bar{\Lambda} \subset \tilde{\Lambda}$ where the  scattering map $S_{\Gamma}$ is well-defined, i.e. we will find a set $\Gamma$ consisting of homoclinic points for which the transversality conditions (\ref{trans}), (\ref{transf}) are satisfied - then the projection $\pi^u(\Gamma)$ is the set $\bar\Lambda$. We will derive perturbatively an explicit formula for $S_{\Gamma}$ up to $\varepsilon O(\varepsilon + |\delta|)$ and then provide an approximation to it by a time-$\varepsilon$ shift along a level curve of a certain Hamiltonian $H_{out}$. In the process, we will also prove Theorems 1.2 and 1.3. 

To study the properties of the scattering map, we need to determine whether the perturbed stable and unstable manifolds $W^{s,u}_{i}(\tilde{\Lambda})$ of $\tilde{\Lambda}$ intersect transversally, and furthermore, if these intersections are strongly transverse.  We will use Melnikov-type method to determine the existence of transverse intersections of $W^{s,u}_{i}(\tilde{\Lambda})$. Since unperturbed $W^{s,u}_{1}(\tilde{\Lambda})$ and $W^{s,u}_{2}(\tilde{\Lambda})$ are symmetric to each other, we will study $W^{s,u}_{2}(\tilde{\Lambda})$ only. From (\ref{invmanifoldsfull2}), the unperturbed invariant manifolds $W^{s,u}_{2}(\tilde{\Lambda})$ coincide and form a $3$-dimensional (unperturbed) homoclinic manifold $W_{2}$ given by
\begin{equation}
W_{2}(\tilde{\Lambda}) = \{(\varphi, \theta, E,t):\varphi \in [0,\pi) \ \frac{\pi}{2}, \theta \in (0,\pi), E \in [E_{1}, E_{2}], t \in \mathbb{S}^1; \quad \sin(\varphi) = \frac{b(t)}{c(t)\tan(\theta)} \ \}.
\label{w22}
\end{equation}
We note that $W^{s,u}_{2}(\tilde{\Lambda})$ are co-dimension $1$ manifolds. Therefore, in order to determine whether they split or not for nonzero $\varepsilon, \delta \neq 0$. we only need one measurement in the normal direction to the coincident tangent spaces of unperturbed $W^{s,u}_{2}(\tilde{\Lambda})$.  Later in this section we will introduce the distance function $\bar{d}(x_{0}, \varepsilon, \delta)$ that measures the splitting of $W^{s,u}_{2}(\tilde{\Lambda})$ for $\varepsilon, \delta \neq 0$ in the normal direction to $x_{0}  = (\varphi_{0}, \theta_{0}, E_{0}, t_{0}) \in W_{2}(\tilde{\Lambda})$ and will show that if $\frac{\partial \bar{d}}{\partial \varphi_0} \neq 0$, then the strong transversality condition (\ref{transf}) is satisfied.
  
Let us consider strong stable and unstable foliations $W^{ss,uu}_{(\bar{E}, \bar{t}) \in \tilde{\Lambda};2}$ of the 
stable/unstable manifolds $W^{s,u}_{2}(\tilde{\Lambda})$ (with notation as in section 3.2). Since we study $W^{s,u}_{2}(\tilde{\Lambda})$, we will drop the subscript $2$ and write $W^{ss,uu}_{(\bar{E}, \bar{t}) \in \tilde{\Lambda}}$ from now on. 

Since $B_{0}$ is integrable and is identity in $(E,t)$ variables, the unperturbed one-dimensional strong stable and unstable fibers $W^{ss,uu}_{(\bar{E}, \bar{t}) \in \tilde{\Lambda}}$  with the same basepoint $(\bar{E}, \bar{t})$ coincide, and may be trivially expressed as graphs over $\varphi$ variable with $E = \bar{E} = \mbox{const}$ and $t = \bar{t} = \mbox{const}$, with $\theta$ as a function of $\varphi$ determined from (\ref{w22}). Because the
spectral gap is large in our situation, the strong stable and unstable fibers are $C^{r-1}$-smooth functions of parameters \cite{fenichel1971persistence}. Therefore, for small $\varepsilon$ and $\delta$, the fibers  $W^{ss,uu}_{(\bar{E}, \bar{t}) \in \tilde{\Lambda}}$ and perturbed stable and unstable manifolds $W^{s,u}(\tilde{\Lambda})$ may be written as 
 \begin{equation}
E = E^{ss,uu}(\varphi; \bar{E}, \bar{t}, \varepsilon, \delta) = \bar{E} + \varepsilon \phi_{1}^{ss,uu}(\varphi; \bar{E}, \bar{t}, 0 ,0) + \varepsilon O(\varepsilon + |\delta|),
\label{fiberE}
\end{equation}
\begin{equation}
t = t^{ss,uu}(\varphi; \bar{E}, \bar{t}, \varepsilon, \delta)  = \bar{t} + \varepsilon \phi_{2}^{ss,uu}(\varphi; \bar{E}, \bar{t}, 0, 0) + \varepsilon O(\varepsilon + |\delta|),
\label{fibert}
\end{equation}
\begin{equation}\label{fiberthet}
\theta = \theta^{s,u}(\varphi, E, t, \varepsilon, \delta),
\end{equation}
where $\phi_{1,2}^{ss,uu}$ are certain $C^{r-1}$ functions, and the graphs of the $C^{r}$-functions $\theta^{s,u}$
define sufficiently large pieces of the stable and unstable manifolds $W^{s,u}_{2}(\tilde{\Lambda})$ at small $\varepsilon$ and $\delta$.

Observe that in the formulas (\ref{fiberE}), (\ref{fibert}) the $O(\delta)$ terms do not appear, since the billiard map is
identity in $(E,t)$ at $\varepsilon=0$, so the strong-stable and strong-unstable fibers are given by $E = \mbox{const}$ and $t = \mbox{const}$ at $\varepsilon=0$ even if $\delta\neq 0$. We note that if we know $\varphi$ (and thus $\theta$) as a function of $(E,t)$, then the formulas (\ref{fiberE}), (\ref{fibert}) will provide us with a formula for the scattering map $S_{\Gamma}$.\\ 

In what follows, we start with the analysis of the splitting of $W^{s}_{2}(\tilde{\Lambda})$ and $W^{u}_{2}(\tilde{\Lambda})$ for the case $\delta = 0$, $\varepsilon \neq 0$.

\begin{proof}[Proof of Theorem 1.2.]
Let us set $\delta = 0$ in (\ref{timebound}). It can be shown \cite{baldoma1998poincare} that the distance $d(x_{0}, \varepsilon, 0)$ (measured in the normal direction at the point $x_{0} = (\varphi_{0}, \theta_{0}, E_{0}, t_{0}) \in W_{2}(\tilde{\Lambda})$) between the perturbed invariant manifolds  $W^{s,u}_{2}(\tilde{\Lambda})$ is given by 
$$d(x_{0}, \varepsilon, 0) = \varepsilon M_{1}(x_{0}) + O\left(\varepsilon^{2} \right),$$ 
where $M_{1}$ is the Melnikov function corresponding to the time-dependent ellipse only (without the quartic polynomial part). The following formula for $M_{1}$ is derived in Appendix A:

\begin{equation}
M_{1}(x_{0}) := \sum_{n = -\infty}^{\infty}  { \langle \nabla I(B_{0}(x_{n}), B_{1}(x_{n}) \rangle  },
\label{m1}
\end{equation} 
where $x_{n}=(\varphi_{n},\theta_{n},E_{n},t_{n}) \in W_{2}(\tilde{\Lambda})$ is the orbit of the point $x_{0} \in W_{2}(\tilde{\Lambda})$ under the map $B_{0}$ (i.e., $x_{n}= B^{n}_{0}(x_{0})$),  and $I$ is the first integral given by (\ref{integs}). Since the separatrices of the static billiard are one-dimensional, we have $\theta_n = \theta^0(\varphi_n)$, a smooth function of $\varphi$ given by (\ref{w22}). Also observe that $E_{n}=E_{0}$, $t_{n}=t_{0}$. Hence we evaluate the series (\ref{m1}) with slow variables held constant and equal to their initial value.

Let us set $\bar{d}(x_{0}, \varepsilon, 0) = \frac{d(x_{0}, \varepsilon, 0)}{\varepsilon}$ for $\varepsilon \neq 0$, and $\bar{d}(x_{0}, 0, 0) = M_{1}(x_{0})$. Then $\bar{d}(x_{0}, \varepsilon, 0) = M_{1}(x_{0}) + O(\varepsilon)$ and the zeroes of $\bar{d}(x, \varepsilon, 0)$ correspond to intersection of $W^{s}_{2}(\tilde{\Lambda})$ and $W^{u}_{2}(\tilde{\Lambda})$. The implicit function theorem implies that if $M_{1}(x_{0})=0$ and $DM_{1}(x_{0}) \neq 0$, then  $W^{s,u}_{2}(\tilde{\Lambda})$ intersect transversally at $x_{0}$ along a $2$-dimensional homoclinic manifold.

We are able to compute $M_{1}$  analytically (see Appendix B). Note that since  we express $\varphi$ and $\theta$ through $\xi$ on homoclinic manifolds using parametrisation (\ref{parameterisation}), effectively we have $M_{1}(x) = M_{1}(\varphi, \theta(\varphi),E,t) = M_{1}(\xi,E,t)$. 

We define $h = \log \lambda$ where $\lambda$ from (\ref{evalue}) is the largest eigenvalue of the linearization of the static map $B_{s}$ at the saddle point $z$. For the map $B$, the value of $\lambda$ depends on $t$ and is given by $\lambda = \lambda(t) = \frac{a+c}{a-c}$ with $a=a(t)$, $b=b(t)$ and $c = c(t)$. We also define the variable $\tau$ in terms of $\xi$ from (\ref{parameterisation}): $\exp(\tau):= \xi := \tan (\varphi/2) $. The method developed in \cite{delshams1996poincare} enables us to compute the sum of the series (\ref{m1}) in terms of elliptic functions - the computations are in Appendix B. Using the definitions of the elliptic functions $\dn$, $\cn$, $\sn$ and complete elliptic integrals $E, E', K, K'$ depending on parameter $m \in [0,1]$ whose dependence on $h$ is $\frac{K'}{K} = \frac{\pi}{h}$ as in Appendix B, we find:
\begin{equation}
M_{1} =  \frac{4b}{v}\left(-\dot{a}b + \dot{b}a\right)\left(\frac{2K}{h}\right)^{2}\left(\frac{E'}{K'} - 1 +\dn^{2}\left(\frac{2K\tau}{h}\right)\right),
\label{melnikovellipse}
\end{equation}
where $v = \sqrt{2E}$ is the (rescaled) speed of the particle. Since $\left(\frac{2K}{h}\right)^{2} \left(\frac{E'}{K'} - 1 +\dn^{2}\left(\frac{2K\tau}{h}\right)\right) = \sum_{n= -\infty}^{n= \infty} { \sech^{2}(\tau + nh) } >0$, as known from \cite{delshams1997melnikov}, the zeroes of $M_{1}$ only exist for values of time $t^{*}$ satisfying $(-\dot{a}b + \dot{b}a)=0$, i.e. when  $\frac{d}{dt} \left(\frac{a}{b}\right) = 0$. For such $t^{*}$, the values of $v$ and $\varphi = \varphi(\tau)$ for which the Melnikov function vanishes can be arbitrary. Hence the zeroes are of the form $(\varphi, \theta(\varphi), t^{*}, v)$. By implicit function theorem, if $DM_{1} \neq 0$, then the zeros of the Melnikov function correspond to zeros of the splitting function $d$ for all small $\varepsilon$, hence to the homoclinic intersections, and these intersections are transverse. The condition $DM_{1} \neq 0$ implies $\frac{d^{2}}{dt^{2}} \left(\frac{a(t^{*})}{b(t^{*})}\right) \neq 0$ which is satisfied if (and only if) $t^{*}$ is a nondegenerate critical point of $\frac{a(t^{*})}{b(t^{*})}$.
\end{proof}

This gives us the existence, for all small $\varepsilon$, of a transverse intersection $\tilde\Gamma$ of $W^u(\tilde\Lambda)$ and $W^s(\tilde\Lambda)$ along a smooth two-dimensional surface close to the surface $\{t=t^*, \theta=\theta^{s,u}_{0,0}(\varphi, E, t^{*})\}$.
Since $\varepsilon$ is just an energy scaling parameter, we obtain the transverse homoclinic intersection $\Gamma$
of $W^u(\Lambda)$ and $W^s(\Lambda)$ for all sufficiently large values of the non-rescaled energy $\mathcal{E}$.

In the limit $\varepsilon=0$ (i.e., in the limit $\mathcal{E}\to+\infty$) the projection $\pi^u$ of this homoclinic surface to the cylinder $\Lambda$ by the strong-unstable fibers, i.e., the domain of definition of the scattering map $S_{\tilde\Gamma}$
shrinks to just the vertical line $t=t^*$ (recall that the strong-unstable fibers are close to the lines $(E,t)=const$ at small $\epsilon$). In other words, we cannot define the scattering map at $\delta = 0$ by first order expansion in $\varepsilon$.
We, therefore, proceed to the case $\delta\neq 0$.

Let us introduce the Melnikov function $M_{2}$ corresponding to the quartic polynomial perturbation $O(\delta)$ \textit{only}, with the time and speed variables frozen. The following formula for $M_{2}$ is again derived in Appendix A:
\begin{equation}
M_{2}(x_{0}) := \sum_{n = -\infty}^{\infty}  { \langle \nabla I(B_{0}(x_{n}), B_{2}(x_{n}) \rangle},
\label{m2}
\end{equation}
where the series is evaluated as before over the orbit of $x_{0} \in W_{2}(\tilde{\Lambda})$ under the map $B_{0}$ and $I$ is the first integral given by (\ref{integs}). Using the same notation as for $M_{1}$ above, we find from Appendix B that (see also \cite{delshams1996poincare}):
\begin{equation}
M_{2} = -4m \frac{ab^{2}}{c^2} \left(\frac{2K}{h}\right)^{3} \dn \left(\frac{2K\tau}{h}\right) \sn \left(\frac{2K\tau}{h}\right) \cn \left(\frac{2K\tau}{h}\right).
\label{melnikdelta}
\end{equation}
It is known \cite{delshams1996poincare} that for each fixed $t$, the function $M_{2}$ is $h$-periodic in $\tau$ and has two simple, in terms of $\tau$, zeroes in the period $[0,h)$.

Before we proceed to prove Theorem 1.3, let us introduce some notation. Rewrite (\ref{melnikovellipse}) as $M_{1} = \frac{f(t)g(\tau,t)}{v}$, and  (\ref{melnikdelta}) as $M_{2} = j(\tau,t)$, where
\begin{align}\label{fgj}
\begin{split}
f(t) =  -\dot{a}b + \dot{b}a,
\\
g(\tau,t) = 4b\left(\frac{2K}{h}\right)^{2}\left(\frac{E'}{K'} - 1 +\dn^{2}\left(\frac{2K\tau}{h}\right)\right) > 0,
\\
j(\tau,t) = -4m \frac{ab^{2}}{c^2} \left(\frac{2K}{h}\right)^{3} \dn \left(\frac{2K\tau}{h}\right) \sn \left(\frac{2K\tau}{h}\right) \cn \left(\frac{2K\tau}{h}\right).
\end{split}
\end{align}

Let us also define the function $\phi(t)$ by:
\begin{equation}
\phi(t) = \min_{\tau} \left| \frac{g(\tau,t)}{j(\tau, t)}\right| /  \sqrt{2}.
\label{phi}
\end{equation}
Note that it is seen from the properties of elliptic functions $\dn$, $\cn$ and $\sn$ that for each $t$, the function $\frac{g(\tau,t)}{j(\tau, t)}$ is non-zero, odd and periodic in $\tau$, with one maximum and one minimum point in the period $\tau \in [0, h(t))$.

\begin{proposition} Let $(v,t)$ be such that $\frac{\varepsilon |f(t)|}{| \delta| v} < \max_{\tau} |\frac{j(\tau, t)}{g(\tau, t)}| -\tilde{k}$ where $\tilde{k}>0$ is constant and $v=\sqrt{2E}$ is the rescaled particle speed. Let $\bar{\Lambda}$
be the subset of $\tilde{\Lambda}$ corresponding to these values of $(E,t)$. Then the stable and unstable manifolds $W^{s,u}_{2}(\bar{\Lambda})$ have a strong-transverse intersection along a $2$-dimensional homoclinic manifold 
$\Gamma$ for all small $\varepsilon, \delta \neq 0$ such that $|\delta|  \gg \varepsilon^{2}$.
\end{proposition}

\begin{proof} As shown in Appendix A, the distance $d(x_{0}, \varepsilon, \delta)$ measured in the normal direction at the point $x_{0} = (\varphi_{0}, \theta_{0}, E_{0}, t_{0}) \in W_{2}(\tilde{\Lambda})$ between the perturbed manifolds $W^{s,u}_{2}(\tilde{\Lambda})$ is given by 
\begin{equation}
d(x_{0}, \varepsilon, \delta) = \varepsilon M_{1}(x_{0}) + \delta M_{2}(x_{0}) + O\left(\varepsilon^{2} + \delta^{2}\right).
\label{modifieddist}
\end{equation}

To deal with two small parameters $\varepsilon$ and $\delta$, let us consider
\begin{equation}
\frac{d(x_{0}, \varepsilon, \delta)}{\delta} = \frac{\varepsilon}{\delta} M_{1}(x_{0})  +   M_{2}(x_{0}) + O(\frac{\varepsilon^{2}}{\delta} + \delta) =  \frac{ \varepsilon f(t)g(\tau,t)}{\delta v} +  j(\tau,t) + O\left(\frac{\varepsilon^{2}}{\delta} + \delta\right).
\label{mm}
\end{equation}

We have from the assumption of the proposition that $\frac{\varepsilon M_{1}(x_{0})}{\delta}$ is bounded for all small nonzero $\delta$ and that terms of order $O\left(\frac{\varepsilon^{2}}{\delta} + \delta \right)$ are uniformly small. Then, if we drop the $O\left(\frac{\varepsilon^{2}}{\delta} + \delta \right)$ terms in (\ref{mm}) and define 
\begin{equation}
\bar{d} = \frac{ \varepsilon f(t)g(\tau,t)}{\delta v} +  j(\tau,t),
\label{melnikovnotation}
\end{equation}
it follows from the implicit function theorem that if $\frac{\partial \bar{d}}{\partial \tau}|_{\bar{d}=0} \neq 0$, then the zeroes of $\bar{d}$ correspond to a transverse intersection of $W^{s,u}_{2}(\tilde{\Lambda})$ along a $2$-dimensional homoclinic manifold $\Gamma$ for all small $\frac{\varepsilon^{2}}{\delta}$ and $\delta$. This manifold is a graph of a smooth function
$\tau$ of $(v,t)$, i.e., it is a graph of a smooth function $\varphi$ of $(E,t)$. Since the strong-stable and strong-unstable
fibers are $\varepsilon$-close to $(E,t)=const$, we immediately have that the conditions (\ref{trans}),(\ref{transf}) of 
the strong transverse intersection are fulfilled at the points of $\Gamma$.

Thus, to prove the proposition, we need to investigate for which values of $v$ and $t$ the function $\bar{d}(\tau)$ has 
simple zeroes. Rearranging (\ref{melnikovnotation}) yields
\begin{equation}
\frac{\varepsilon f(t)}{\delta v} = -\frac{j(\tau, t)}{g(\tau, t)}.
\label{inters}
\end{equation}
Let us fix any $t$ and $v$. Since $g(\tau,t) >0$ and $v>0$, solutions of this equation belong to the interval of values of $\tau$ for which the sign of $j(\tau, t)$ is the same as the sign of $f(t)$. Degenerate zeros correspond to critical points of 
the function  $\mbox{fix}_{t} \frac{j(\tau,t)}{g(\tau,t)}$. It can be checked that it has exactly one (positive) maximum
and one (negative) minimum in the period $\tau \in [0, h(t))$; since this function is odd, the maximum and minimum are the same in the absolute value. It follows that for any constant $\tilde{k}>0$, given $t$ and $v$ such that
\begin{equation}
\frac{\epsilon |f(t)|}{| \delta | v} < \max_{\tau} |\frac{j(\tau, t)}{g(\tau, t)}| -\tilde{k},
\label{vtset}
\end{equation}
equation (\ref{inters}) has exactly two solutions $\tau(t,v)$, both of them non-degenerate.

As we explained above, these solutions correspond to strong-transverse homoclinic intersections, provided $\delta$, $\varepsilon$ and $\varepsilon^2/\delta$ are small enough.
\end{proof}

Now we can prove Theorem 1.3.
\begin{proof}
We have shown in Proposition 4.3 that the values of $(v,t)$ satisfying (\ref{vtset}) correspond to a strong-transverse
homoclinic intersection, i.e., they lie in the domain of definition of the scattering map $S_{\Gamma}$.
It is also seen from (\ref{inters}) that if, for some $\tilde k>0$,
\begin{equation}
\frac{\varepsilon |f(t)|}{| \delta | v} > \max_{\tau} |\frac{j(\tau, t)}{g(\tau, t)}| + \tilde{k},
\label{vtset0}
\end{equation}
then the function $d$ does not have zeros for small $\delta$, $\frac{\varepsilon^{2}}{\delta}$ and $\varepsilon$, so 
the region (\ref{vtset0}) is not included in the domain of the scattering map. 

Now, by noting that $\varepsilon/v=1/\sqrt{2\mathcal{E}}$ where $\mathcal{E}$ is the non-rescaled energy,
we immediately obtain the statement of the theorem from estimates (\ref{vtset}),(\ref{vtset0}).
\end{proof}

Now we proceed to obtain a first order expression in perturbation parameters $\varepsilon, \delta$ for $S_{\Gamma}$ in terms of $E$ and $t$. Let $(\bar{E}_{0}, \bar{t}_{0}) \in  \bar{\Lambda}$ and $(\tilde{E}_{0}, \tilde{t}_{0}) \in \bar{\Lambda}$ be two points in the domain of $S_{\Gamma}$ given by (\ref{eqbarlambda+}). We will obtain a perturbative expression up to $\varepsilon O(\varepsilon + |\delta|)$ for the scattering map $S_{\Gamma}: \bar{\Lambda} \mapsto \bar{\Lambda}$:
\begin{equation}
S_{\Gamma}: (\bar{E}_{0}, \bar{t}_{0}; \varepsilon, \delta ) \mapsto (\tilde{E}_{0}, \tilde{t}_{0}; \varepsilon, \delta).
\label{scatteringexact}
\end{equation}
We will call the first order approximation of $S_{\Gamma}$ the truncated scattering map. 

Let us derive a formula for $\phi^{ss,uu}_{1,2}$ given by (\ref{fiberE}), (\ref{fibert}). Let us take a point $(\varphi_{0}, \theta_{0}, E_{0}, t_{0}) \in \Gamma \pitchfork W^{ss}_{(\bar{E}_{0}, \bar{t}_{0})}$. Its energy component $E_{0}$ is given by (\ref{fiberE}):
$$E_{0} = E^{ss}(\varphi_{0}; \bar{E}_{0}, \bar{t}_{0}, \varepsilon, \delta) = \bar{E}_{0} + \varepsilon \phi_{1}^{ss}(\varphi_{0}; \bar{E}_{0}, \bar{t}_{0},0, 0) + \varepsilon O(\varepsilon + |\delta|).$$

Let $B \left(W^{ss}_{(\bar{E}_{0}, \bar{t}_{0})} \right)$ denote the action of map $B$ on the point $(\varphi_{0}, \theta_{0}, E_{0}, t_{0})$ on the leaf $W^{ss}_{(\bar{E}_{0}, \bar{t}_{0})}$.  Let $(\varphi_{n, \varepsilon, \delta}, \theta_{n, \varepsilon, \delta}, E_{n, \varepsilon, \delta}, t_{n, \varepsilon, \delta})$ be the orbit of the point $(\varphi_{0}, \theta_{0}, E_{0}, t_{0})$ under $B$, and let $(\bar{E}_{n, \varepsilon}, \bar{t}_{n, \varepsilon})$ be the orbit of $(\bar{E}_{0}, \bar{t}_{0}) \in \bar{\Lambda}$ under $B$ (the notation is the same as in (\ref{components})). Note that on $\bar{\Lambda}$ we have $(\varphi_{n, \varepsilon}, \theta_{n, \varepsilon}) = (0, \frac{\pi}{2})$  for all $n$, and there is no dependence on $\delta$ of $B$ restricted to $\Lambda$.

Consider the energy $E$ component of $B \left(W^{ss}_{(\bar{E}_{0}, \bar{t}_{0})} \right)$.  Using the notation (\ref{components}) gives
\begin{equation}
\begin{split}
E_{1, \varepsilon} = E_{1, 0} + \varepsilon f_{3}(\varphi_{0}, \theta_{0}(\varphi_{0}), E_{0}, t_{0}) + \varepsilon O(\varepsilon + |\delta|)
\\
=E_{0} + \varepsilon f_{3}(\varphi_{0}, \theta_{0}(\varphi_{0}),  E_{0}, t_{0}) + \varepsilon O(\varepsilon + |\delta|)
\\
=\bar{E}_{0} +  \varepsilon \phi_{1}^{ss}(\varphi_{0}(\varphi_{0}), \bar{E}_{0}, \bar{t}_{0}, 0) + \varepsilon f_{3}(\varphi_{0}, \theta_{0}(\varphi_{0}), E_{0}, t_{0}) + \varepsilon O(\varepsilon + |\delta|),
\end{split}
\label{invfol}
\end{equation}
where $f_{3}(\varphi_{0}, \theta_{0}(\varphi_{0}), E_{0}, t_{0})$ is evaluated on the unperturbed homoclinic trajectory on the unperturbed homoclinic manifold $W_{2}(\bar{\Lambda})$ given by (\ref{w22}). The invariance property of the stable and unstable foliations gives $B \left(W^{ss}_{(\bar{E}_{0}, \bar{t}_{0})} \right) = W^{ss}_{B\left(\bar{E}_{0}, \bar{t}_{0}\right)}$ and, by definition of the map $B$, we have $W^{ss}_{B\left(\bar{E}_{0}, \bar{t}_{0})\right)} = W^{ss}_{(\bar{E}_{1, \varepsilon}, \bar{t}_{1, \varepsilon})}$. Therefore, we have 
\begin{equation}
\begin{split}
 E_{1, \varepsilon} = E(\varphi_{1, \varepsilon}; \bar{E}_{1, \varepsilon}, \bar{t}_{1, \varepsilon}, \varepsilon, \delta)
 \\
 = \bar{E}_{1, \varepsilon} + \varepsilon \phi_{1}^{ss}(\varphi_{1}, \bar{E}_{1}, \bar{t}_{1}, 0,0) + \varepsilon O(\varepsilon + |\delta|)
 \\
 = \bar{E}_{0} + \varepsilon f_{3}(\cdot, \bar{E}_{0}, \bar{t}_{0})|_{\bar{\Lambda}} +  \varepsilon \phi_{1}^{ss}(\varphi_{1}, \bar{E}_{0}, \bar{t}_{0}, 0, 0) + \varepsilon O(\varepsilon + |\delta|),
\end{split}
\label{invcyl}
\end{equation}
where we used $\bar{E}_{1} = \bar{E}_{0}$, $\bar{t}_{1} = \bar{t}_{0}$ and $\bar{E}_{1, \varepsilon} = \bar{E}_{0} + \varepsilon f_{3}(\cdot, \bar{E}_{0}, \bar{t}_{0})|_{\bar{\Lambda}}$ in the last line above, with $f_{3}(\cdot,\bar{E}_{0}, \bar{t}_{0})| _{\Lambda}$ denoting the restriction of $f_{3}$ to $\bar{\Lambda}$. The notation $f_{3}(\cdot,\bar{E}_{0}, \bar{t}_{0})| _{\bar{\Lambda}}$ signifies that the variables $(\varphi, \theta)$ on $\bar{\Lambda}$ are fixed at $(0, \pi/2)$. 

Examining the coefficients of $O(\varepsilon)$ terms in (\ref{invfol}) and (\ref{invcyl}) gives 
\begin{equation}
\phi_{1}^{ss}(\varphi_{0}; \bar{E}_{0}, \bar{t}_{0},0, 0) = \phi_{1}^{ss}(\varphi_{1}; \bar{E}_{0}, \bar{t}_{0}, 0, 0) + f_{3}(\cdot , \bar{E}_{0}, \bar{t}_{0})|_{\Lambda} - f_{3}(\varphi_{0},\theta_{0}(\varphi_{0}), E_{0}, t_{0}).
\label{recurrence}
\end{equation}
Upon iterating (\ref{recurrence}), we obtain an expression for $\phi^{ss}_{1}(\varphi_{0}; \bar{E}_{0}, \bar{t}_{0},0, 0)$:
\begin{equation}
\phi^{ss}_{1}(\varphi_{0}; \bar{E}_{0}, \bar{t}_{0},0, 0) = \phi^{ss}_{1}(\varphi_{\infty}; \bar{E}_{\infty}, \bar{t}_{\infty},0, 0) + \sum_{i=0}^{\infty} \{ f_{3}(\cdot, \bar{E}_{i}, \bar{t}_{i})|_{\Lambda}- f_{3}(\varphi_{i}, \theta_{i}(\varphi_{i}), E_{i}, t_{i})|_{W_{2}(\bar{\Lambda})} \},
\end{equation}
where $\varphi_{\infty} = 0$, $\bar{E}_{\infty} = \bar{E}_{0}$ and $\bar{t}_{\infty} = \bar{t}_{0}$, and $f_{3}(\varphi_{i}, \theta_{i}(\varphi_{i}), E_{i}, t_{i})$ is evaluated over the unperturbed homoclinic trajectory on $W_{2}(\bar{\Lambda})$ given by (\ref{w22}). Therefore the slow variables $(E,t)$ in the summation are held constant with $\bar{E} = E_{i} = E_{0}, \bar{t} = t_{i} = t_{0}$. Hence, the equation for the energy $E$ component of the strong stable leaf $W^{ss}_{(\bar{E}_{0}, \bar{t}_{0})}$ through $(E_{0}, t_{0})$ is
\begin{equation}
E_{0} = \bar{E}_{0} + \varepsilon \phi^{ss}_{1}(0;\bar{E}_{0}, \bar{t}_{0},0, 0) + \varepsilon \sum_{i=0}^{\infty} \{ f_{3}(\cdot , \bar{E}_{i}, \bar{t}_{i})|_{\Lambda}- f_{3}(\varphi_{i}, \theta_{i}(\varphi_{i}), E_{i}, t_{i})|_{W_{2}(\bar{\Lambda})} \} +
\varepsilon O(\varepsilon + |\delta|).
\label{bare}
\end{equation}
\
Similarly, we express the $E$ component of the strong unstable leaf $W^{uu}_{(\tilde{E}_{0}, \tilde{t}_{0})}$ through $(\varphi_{0}, \theta_{0}, E_{0}, t_{0}) \in \Gamma \pitchfork W^{uu}_{(\tilde{E}_{0}, \tilde{t}_{0})}$  with the base point $(\tilde{E}_{0}, \tilde{t}_{0})$ as
\begin{equation}
E_{0} = \tilde{E}_{0} + \varepsilon \phi^{uu}_{1}(0;\tilde{E}_{0}, \tilde{t}_{0},0, 0) - \varepsilon \sum_{i=- \infty}^{-1} \{ f_{3}(\cdot, \tilde{E}_{i}, \tilde{t}_{i})|_{\Lambda} - f_{3}(\varphi_{i}, \theta_{i}(\varphi_{i}),  E_{i}, t_{i})|_{W_{2}(\bar{\Lambda})} \} + \varepsilon O(\varepsilon + |\delta|).
\label{tilden}
\end{equation}

Using that $\phi^{ss}_{1}(0;\bar{E}_{0}, \bar{t}_{0},0, 0) -\phi^{uu}_{1}(0;\tilde{E}_{0}, \tilde{t}_{0},0, 0)=0$ since the unperturbed stable and unstable foliations coincide, and subtracting  (\ref{tilden}) from (\ref{bare}) gives an expression for  the $E$ component of the truncated scattering map up to $ \varepsilon O(\varepsilon + |\delta|)$:
\begin{equation}
\tilde{E}_{0} = \bar{E}_{0} + \varepsilon \sum_{i=- \infty}^{ \infty} \{ f_{3}(\cdot, E_{i}, t_{i})|_{\Lambda} - f_{3}(\varphi_{i}, \theta_{i}(\varphi_{i}), E_{i}, t_{i})|_{W_{2}(\bar{\Lambda})} \}.
\label{scatteringE}
\end{equation}
\
Analogously, for the $t$ components of the truncated map $S_{\Gamma}$ we obtain:
\begin{equation}
\tilde{t}_{0} = \bar{t}_{0} + \varepsilon \sum_{i = - \infty}^{ \infty} \{f_{4}(\cdot, E_{i}, t_{i})|_{\Lambda} - f_{4}(\varphi_{i}, \theta_{i}(\varphi_{i}), E_{i}, t_{i})|_{W_{2}(\bar{\Lambda})} \}.
\label{scatteringt}
\end{equation}

We see that $S_{\Gamma}$ is close to identity and may be approximated by a time-$\varepsilon$ shift of some Hamiltonian flow 
$H_{out}(t,E)$ with the accuracy $\varepsilon O(\varepsilon + |\delta|)$. From (\ref{scatteringE}) and (\ref{scatteringt}), we have Hamilton's equations for $H_{out}(t, E)$ with $' = \frac{d}{ds}$ denoting differentiation with respect to auxiliary time $s$:

$$E' = \sum_{i=-\infty}^{\infty} \{ f_{3}(\cdot, E_{i}, t_{i})|_{\Lambda} - f_{3}(\varphi_{i}, \theta_{i}(\varphi_{i}), E_{i}, t_{i})|_{W_{2}(\bar{\Lambda})} \}, \qquad t' =   \sum_{i=-\infty}^{\infty} \{  f_{4}(\cdot, E_{i}, t_{i})|_{\Lambda} -f_{4}(\varphi_{i}, \theta_{i}(\varphi_{i}), E_{i}, t_{i})|_{W_{2}(\bar{\Lambda})} \}.$$

Let us compute the above infinite sums. Recalling the definitions of $f_{3}$ and $f_{4}$ given by (\ref{f3}, \ref{f4}) and that the summation is performed over the fast variables $(\varphi, \theta)$ while the slow variables $(E,t)$ are held fixed at their initial value $(E_{i}=E_{0}, t_{i} = t_{0})$,  it turns out (see Appendix C) that the sum for $E'$ may be computed to give

\begin{equation}
\sum_{i=-\infty}^{\infty}  \{ f_{3}(\cdot, E_{i}, t_{i})|_{\Lambda} - f_{3}(\varphi_{i}, \theta_{i}(\varphi_{i}), E_{i}, t_{i})|_{W_{2}(\bar{\Lambda})} \} = 2\sqrt{2E_{0}} \sum_{i = -\infty}^{\infty}  \{ {- \dot{a}(t_{i+1}) + u_{i+1} \sin (\theta_{i+1}) } \} = 2 \sqrt{2E_{0}}\left(\frac{b\dot{b}-a\dot{a}}{c}\right),
\label{Efib}
\end{equation}
where $u_{i+1}$ is the normal speed of the boundary at $(i+1)$-th impact. The expression for $t'$ may be computed using the geometric property of the ellipse \cite{ramirez2005exponentially} that the sum of homoclinic lengths converges to $-2c(t_{0})$:
\begin{equation}
\sum_{i=-\infty}^{\infty}  \{ f_{4}(\cdot, E_{i}, t_{i})|_{\bar{\Lambda}} - f_{4}(\varphi_{i}, \theta_{i}(\varphi_{i}), E_{i}, t_{i})|_{ W_{2}(\bar{\Lambda})} \} = \frac{1}{\sqrt{2E_{0}}} \sum_{i = -\infty}^{\infty} \{ {2a(t_{i}) - D_{0}|_{ W_{2}(\bar{\Lambda})} } \} = \frac{2c(t)}{\sqrt{2E}},
\label{tfiber}
\end{equation}
where $D_{0}|_{ W_{2}(\tilde{\Lambda})}$ is flight distance  $D_{0}$ (see (\ref{euclidzero})) evaluated on  $W_{2}(\tilde{\Lambda})$. Therefore the scattering map $S_{\Gamma}$ is approximated up to $O(\varepsilon(\varepsilon + |\delta|))$ by a time-$\epsilon$ shift along a trajectory of the solution of the differential equation
$$\frac{dE}{dt} = -2E \frac{\dot{c}(t)}{c(t)}.$$
\
This corresponds to  the Hamiltonian vector field
\begin{equation}
t'= \frac{\partial H_{out}}{\partial E}  = \frac{\sqrt{2}c(t)}{\sqrt{E}}, \qquad E'  = -\frac{\partial H_{out}}{\partial E} = -2\sqrt{2E}\dot{c}(t).
\label{hamiltonsouteq}
\end{equation}
\
Therefore we obtain the following Hamiltonian $H_{out}(t,E)$ defined on $\Lambda$ that approximates $S_{\Gamma}$:
\begin{equation}
H_{out}(t,E) = 2\sqrt{2E}c(t).
\label{outh}
\end{equation}

\section{Energy growth}

In \cite{gelfreich2017arnold} it is proved (Lemma 4.4) that if two points on a normally-hyperbolic invariant manifold
of a symplectic diffeomorphism are connected by an orbit of the iterated function system (IFS) formed by the inner and scattering maps, then there exists a trajectory of the original diffeomorphism that connects arbitrarily small neighbourhoods of those two points. Lemmas 3.11 and 3.12 of \cite{gidea2014general} show that the same is true when orbits of an IFS are infinite (in one direction).

In this section we prove Theorem 1.1 using these facts. Namely, we consider an iterated function system $\{\Phi, S_{\Gamma}\}$ 
comprised of the inner map $\Phi$ defined by (\ref{innerscaled}) and the scattering map $S_{\Gamma}$ defined by (\ref{scatteringexact}),
and show that it has an orbit with the energy $\mathcal{E}$ tending to infinity. By the above quoted results, the  existence of the orbit of the map $B$ for which the energy grows to infinity follows too.

We start with the analysis of the behaviour of the IFS in the rescaled coordinates.
\begin{lemma}
For any initial condition $(E,t) \in \tilde{\Lambda}$ with $|\delta| \gg \varepsilon^{2}$, there exist positive integers 
$n_{1}, n_{2}$, where $n_{1} + n_{2}  = O\left(\frac{1}{\varepsilon}\right)$, such that 
the gain $\Delta H_{in}$ of the Hamiltonian $H_{in}(t,E; \varepsilon)$ (given in Remark 4.2) along the orbit
$S_{\Gamma}^{n_2} \circ \Phi^{n_1}$ of the IFS 
$\{\Phi), S_{\Gamma}\}$ is $\Delta H_{in} \geq K_{1} \min \left(\frac{\delta^2}{\varepsilon^2}, 1 \right)$ where $K_{1}$ is a strictly positive constant. 
\label{energylemma}
\end{lemma}

\begin{proof}

Suppose that $t^{*}$ is a nondegenerate critical point of $\frac{a}{b}$. Then $f(t^{*}) = 0$, which corresponds to $\frac{\dot{c}}{c} = \frac{\dot{a}}{a}$, where $f(t)$ is defined in (\ref{fgj}). Therefore $(E,t^{*}) \in \bar{\Lambda}$ where $\bar{\Lambda}$ is the domain of $S_{\Gamma}$ as in (\ref{eqbarlambda+}) for all $E$. There exists an interval in $t$ near $t^{*}$, denote it $[t_{1}, t_{2}]$, such that $\frac{\dot{c}}{c} < \frac{\dot{a}}{a}$ and $S_{\Gamma}$ is defined at  $t \in [t_{1}, t_{2}]$.  Denote $\Delta t = t_{2} - t_{1}$. Let us obtain the lower bound on $\Delta t$.  Taking (\ref{eqbarlambda+}) and rewriting it in scaled variable $v$ using the definition (\ref{scale}), then Taylor expanding about $t^{*}$, we obtain that small $\Delta t$ can be chosen such that
\begin{equation}
v  < \frac{c_{2} \varepsilon \Delta t}{\delta}
\label{appr}
\end{equation} 
for a constant $c_{2} > 0$. In other words, we can always choose the interval $[t_1,t_2]$ such that
$$\Delta t = c_{1}\mbox{min} \left(\frac{\delta}{\varepsilon}, 1 \right)$$
where $c_{1}>0$ is constant. 

Let us take a point $(E,t) \in \tilde{\Lambda}$.  Iterate $(E,t)$ under the inner map $\Phi$ until the image $\Phi^{n_{1}}(E,t) = (E_{n_{1}}, t_{n_{1}})$  enters the domain $[t_{1}, t_{2}]$, i.e., $t_{n_{1}} \in [t_{1}, t_{2}]$, and thus the point $(E_{n_{1}}, t_{n_{1}}) \in \bar{\Lambda}$ (if $t$ is originally in the interval $[t_{1}, t_{2}]$, then iterate until it gets out of this interval and, then, returns to it again).
Note that the change in $t$ during one iteration of $\Phi$ is of order $\varepsilon$ which is much smaller than $\Delta t$ (because $\delta/\varepsilon \gg \varepsilon$ by assumption). Therefore, the iterates of $\Phi$ cannot "miss" $[t_{1}, t_{2}]$ and the number of iterations $n_1$ is bounded from above as $O(\varepsilon^{-1})$.

There exists a level curve $h_{0}$ of $H_{in}(t, E ;\varepsilon)$ passing through $(E,t)$. The orbit of $(E,t)$ under $\Phi$ will follow the level curve $h_{0}$.  Indeed, the number of iterations $n_{1}$ of $\Phi$ is bounded by a number of order $O\left(\frac{1}{\varepsilon} \right)$. Denote by $\phi^{n_{1}\varepsilon}_{h_{0}}$ the time$-n_{1}\varepsilon$ shift along $h_{0}$ with the initial condition $(E,t)$. It follows from standard mean value theorem estimates and the Remark 4.2 that $\Phi^{n_{1}}$ coincides with  $\phi^{n_{1}\varepsilon}_{h_{0}}$ up to $O(\varepsilon^{r})$. Hence we have the following bound for the difference between the level of $H_{in}$ at $(E_{n_{1}}, t_{n_{1}})$  and  $h_{0}=H_{in}(\phi^{n_{1}\varepsilon}_{h_{0}}(E,t); \varepsilon)$:
\begin{equation}
\begin{split}
||H_{in}(E_{n_{1}},t_{n_{1}}; \varepsilon) - H_{in}(\phi^{n_{1}\varepsilon}_{h_{0}}(E,t); \varepsilon)|| & \leq \mbox{max}_{(E,t) \in \tilde{\Lambda}} || DH_{in}(E,t; \varepsilon)|| ||(E_{n_{1}}, t_{n_{1}}) - \phi_{h_{0}}^{n_{1}\varepsilon}(E,t) || \\
& \leq   \mbox{max}_{(E,t) \in \tilde{\Lambda}} || DH_{in}(E,t; \varepsilon)||\tilde{C}_{1} \varepsilon^{r} \\
& < C_{2} \varepsilon^{r},
\end{split}
\label{hinbound}
\end{equation}
where $\mbox{max}_{(E,t) \in \tilde{\Lambda}} || DH_{in}(E,t; \varepsilon)||$ is bounded as $\tilde{\Lambda}$ is compact; $\tilde{C}_{1}, C_{2}>0$ are constants. Since $r\geq 4$, we have that the error in the difference of $H_{in}(t, E;\varepsilon)$ following $\Phi$ is maximum  $O(\varepsilon^4)$.

When  $t \in [t_{1}, t_{2}]$, following the level curve of $H_{out}(t, E; \varepsilon, \delta)$ will give a greater gain of energy than following $H_{in}(t, E;\varepsilon)$. Therefore we iterate $S_{\Gamma}$ while $t \in [t_{1}, t_{2}]$ and its orbit will follow the level curve of $H_{out}(t, E; \varepsilon, \delta)$. Following $H_{in}(t, E;\varepsilon)$ will switch to following $H_{out}(t, E;\varepsilon, \delta)$ when $t = t_{1} + O(\varepsilon)$ and then switch back to following another level curve of $H_{in}(t, E;\varepsilon)$ when $t = t_{2} + O(\varepsilon)$. As $\Delta t\gg\varepsilon$, it follows that the number $n_{2}$ of iterates of $S_{\Gamma}$ is $n_{2} \sim O(\frac{\Delta t}{\varepsilon})$. 

Let us consider the Hamiltonian flow given by $H_{out}(t, E) = 2 \sqrt{2E}c(t)$ as in (\ref{outh}). Since the scattering map $S_{\Gamma}$ and the time-$\varepsilon$ flow map $\phi_{H_{out}}$ of $H_{out}(t, E)$ coincide up to $O(\varepsilon(\varepsilon + |\delta|))$, we have an upper bound for the difference between the values of $H_{in}$ evaluated at $(E_{n_{2}}, t_{n_{2}})$ and at the time-$n_{2} \epsilon$ shift by the flow of $H_{out}$ with the initial condition $(E_{n_{1}}, t_{n_{1}})$ (below $C_{3} >0 $ is an irrelevant constant):
\begin{equation}
||H_{out}(E_{n_{2}},t_{n_{2}}) - H_{out}(\phi^{n_{2}\varepsilon}_{H_{out}}(E_{n_{1}},t_{n_{1}})) || < C_{3} \Delta t (\varepsilon + |\delta|)=
O(|\delta|(\varepsilon+|\delta|)).
\label{houtbound}
\end{equation}
Let us compute the change $ \Delta H_{in}$ following the level curve of $H_{out}(t,E)$. Using (\ref{outh}), we have
$$\Delta H_{in} = \int_{s_{1}}^{s_{2}} \frac{dH_{in}(t, E;\varepsilon)}{ds} \mid _{H_{out}=\mbox{const}} ds =  \int_{t_{1}}^{t_{2}} \frac{dH_{in}(t, E;\varepsilon)}{ds} \mid _{H_{out}=\mbox{const}} \frac{ds}{dt}dt. $$
where $t_{1,2} = t_{1,2}(s_{1,2})$ and $s$ is the auxiliary time variable. Since $\frac{dH_{in}(t, E;\varepsilon)}{ds} = \frac{\partial H_{in}(t, E;\varepsilon)}{\partial E} \frac{dE}{ds} + \frac{\partial H_{in}(t, E;\varepsilon)}{\partial t} \frac{dt}{ds}$ where we use (\ref{hamiltonsouteq}) for $\frac{dE}{ds}$ and $\frac{dt}{ds}$, yielding $\frac{dH_{in}(t, E;\varepsilon)}{ds} = 4(\dot{a}c - \dot{c}a) + O(\varepsilon )$. Thus
$$\Delta H_{in} = \int_{t_{1}}^{t_{2}} \left(4(\dot{a}c - \dot{c}a) + O(\varepsilon ) \right) \frac{ds}{dt}dt  = 
\int_{t_{1}}^{t_{2}} \left[H_{in}(t, E; 0) \left(\frac{\dot{a}}{a} - \frac{\dot{c}}{c}\right) + O(\varepsilon)\right] dt.$$
Note that $\frac{\dot{a}}{a} - \frac{\dot{c}}{c}>0$ for all $t \in [t_{1}, t_{2}]$. We expand the above integral in Taylor series to obtain the estimate from below
\begin{equation}
\Delta H_{in} \sim H_{in}(t,E; 0)(\Delta t)^{2} + O(\varepsilon) \Delta t > K_{1} \mbox{min} \left(\frac{\delta^{2}}{\varepsilon^{2}}, 1 \right),
\label{deltahin}
\end{equation}
where $K_{1}>0$ is constant.

The error estimates (\ref{hinbound}), (\ref{houtbound}) imply that for $n_{1}$ iterations of $\Phi$ followed by $n_{2}$ iterations of $S_{\Gamma}$, the increase of $H_{in}$ is given by  $\Delta H_{in} > K_{1} \mbox{min} \left(\frac{\delta^{2}}{\varepsilon^{2}}, 1 \right) + O\left(\varepsilon^{4}, |\delta| (\varepsilon + |\delta|) \right)$. Since $|\delta|\gg \varepsilon^2$, the net gain of $H_{in}$ is strictly positive and is given by (\ref{deltahin}) indeed.  \end{proof}

\begin{proof}[Proof of Theorem 1.1.]
Take any initial condition $(\mathcal{E},t)$ on the cylinder $\Lambda$, where $\mathcal{E}$ is the non-rescaled kinetic energy of
the billiard particle. If $\mathcal{E} \geq \frac{C}{| \delta |}$ for a sufficiently large constant $C>0$, then one can find the scaling parameter
$\varepsilon$ such that the scaled energy $E=\mathcal{E} \varepsilon^{2}$ lies in the middle of the interval $[E_{1}, E_{2}]$ corresponding to the compact piece $\tilde{\Lambda}$ considered in the lemma above, and $| \delta| \gg \varepsilon^{2}$. Suppose that $\tilde{\Lambda}$ is sufficiently large, $E_{2} - E_{1}$ is sufficiently large. Since the function $a(t)$ is bounded, the ratio of $H_{in}/E$ is bounded away from zero and infinity. Therefore, by repeated application of Lemma 5.1 we find that the iterated function system $\{\Phi, S_{\Gamma} \}$ has an orbit, starting with our initial conditions $(E=(E_1+E_2)/2, t)$ for which the value of $H_{in}(t, E, \varepsilon)$ increases without bound until the orbit stays in $\tilde{\Lambda}$. In other words, this orbit will eventually get to the values of $E$ larger than $E_2$. 

For the non-rescaled energy $\mathcal{E}$ this means the multiplication at least to $2E_2/(E_1+E_2)>1$. Thus, we have shown that
for every initial condition $(\mathcal{E},t)$ with $\mathcal{E} \geq \frac{C}{| \delta |}$ there exists an orbit of the IFS 
with the end point $(\bar{\mathcal{E}},\bar{t})$ such that $\bar{\mathcal{E}}>q \mathcal{E}$ where the factor $q>1$ is independent of the initial point. By taking the end point of the orbit we just constructed as the new initial point, and so on, we continue the process up to infinity and obtain the orbit of the IFS for which the energy $\mathcal{E}$ tends to infinity.

The shadowing Lemmas 3.11, 3.12 of \cite{gidea2014general} imply the existence of a true orbit of the map $B$  that shadows the orbit of the IFS $\{\Phi, S_{\Gamma} \}$, so the energy $\mathcal{E}$ tends to infinity along this true orbit too (the shadowing lemmas of \cite{gidea2014general} require compactness of $\Lambda$, but it is easy to see that the result remains valid also in the situation where every
orbit of the inner map is bounded - so the Poincare recurrence theorem can be used, and this property holds true in our case, as the KAM curves bound every orbit of $\Phi$).
\end{proof}

\subsection*{Acknowledgements}
The authors would like to thank Vassili Gelfreich, Anatoly Neishtadt and Rafael Ramirez Ros for useful discussions. This work was supported by the grant 14-41-00044 of the Russian Science Foundation. Carl Dettmann's research is supported by EPSRC grant EP/N002458/1. Vitaly Fain's research is supported by University of Bristol Science Faculty Studentship grant. Dmitry Turaev's research is supported by EPSRC grant EP/P026001/1.

\appendix

\section{Melnikov function derivation}

In this section we provide a derivation of the Melnikov function $M_{1}$ given by (\ref{m1}); the Melnikov function $M_{2}$ given by (\ref{m2}) is derived in the same manner. Melnikov theory for $n$-dimensional diffeomorphisms with hetero-homoclinic connections to normally hyperbolic invariant manifolds has been developed in \cite{baldoma1998poincare}.  Let us briefly review this construction and adapt it for our slow-fast setup. Since for the map $B$ the invariant manifolds $W^{s,u}(\Lambda)$ are three-dimensional while the phase space is four-dimensional, one only needs a scalar Melnikov function to measure their splitting for small nonzero $\varepsilon, \delta$. Let us consider the case $\varepsilon >0$ and $\delta = 0$. Take $\tilde{\Lambda}$ as in (\ref{subsetnhim}). By symmetry we only need to consider the splitting of $W^{s,u}_{2}(\tilde{\Lambda})$. Theorem 3.4 in \cite{baldoma1998poincare} gives the following expression for the Melnikov function $M_{1}$:

\begin{equation}
\sum_{n= -\infty}^{n= \infty} { \langle DB_{0}^{n}(x_{-n})B_{1}(x_{-n-1}), \nu(x_{0}) \rangle },
\label{baldoma}
\end{equation}
where $x_{0} = (\varphi_{0}, \theta_{0}, E_{0}, t_{0}) \in W_{2}(\tilde{\Lambda})$ and $\nu(x_{0})$ is the vector forming a basis of an orthogonal space to the tangent space of the unperturbed three dimensional homoclinic manifold $W_{2}(\tilde{\Lambda})$. Since $B_{0}$ has a first integral $I$, we take $\nabla I(x_{0}) = \nu(x_{0})$. By the property of first integrals, observe that $\nabla I(x_{0}) = (DB_{0}(x_{0}))^{T} \nabla I(x_{1})$ and by induction $\nabla I(x_{0}) =  (DB_{0}^{n}(x_{0}))^{T} \nabla I(x_{n})$. Then rewriting (\ref{baldoma}) yields

\begin{align}
M_{1} = &\sum_{n= -\infty}^{n= \infty} {\langle DB_{0}^{n}(x_{-n}) B_{1}(x_{-n-1}), \nabla I(x_{0}) \rangle} \\
&= \sum_{n= -\infty}^{n= \infty} {\langle  B_{1}(x_{-n-1}), (DB_{0}^{n}(x_{-n}))^{T}\nabla I(x_{0}) \rangle} \\
&= \sum_{n= -\infty}^{n= \infty} {\langle  B_{1}(x_{n-1}), (DB_{0}^{-n}(x_{n}))^{T}\nabla I(x_{0}) \rangle} \\
& = \sum_{n= -\infty}^{n= \infty} {\langle  B_{1}(x_{n-1}), (DB_{0}^{-n}(x_{n}))^{T} (DB_{0}^{n}(x))^{T}\nabla I(x_{n}) \rangle} \\
&=\sum_{n= -\infty}^{n= \infty} {\langle  B_{1}(x_{n-1}), (DB_{0}^{n}(x) DB_{0}^{-n}(x_{n}))^{T}\nabla I(x_{n}) \rangle} \\
&= \sum_{n =  -\infty}^{n = \infty} {\langle B_{1}(x_{n-1}), \nabla I(x_{n}) \rangle },
\end{align}
\label{star}
which gives (\ref{m1}). Note that since $(E,t)$ are close to identity, effectively the summation above is performed only over the fast variables $(\varphi, \theta)$, while $(E,t)$ are held at an initial value, hence they enter the sum as ``fixed coefficients". For (\ref{m1}) to converge, we require the restriction of the perturbed components of fast variables $(\varphi, \theta)$ (i.e. $f_{1}, f_{2}, g_{1}, g_{2}$) to $\Lambda$ to vanish, and the form of $I$ given by (\ref{integs}) was chosen to ensure that $\nabla_{E,t} I(x) = 0$ on $\Lambda$ (i.e. the gradient of $I$ with respect to slow variables $(E,t)$ ).

The Melnikov function (\ref{m2}) is obtained by repeating the same steps above, by setting $\varepsilon =0$, and $\delta \neq 0$. Note that the $B_{2}$ terms are independent of $E$ and are evaluated at a given fixed moment of time $t$, hence effectively (\ref{m2}) corresponds to the Melnikov function for the $\delta$-polynomial perturbation of the elliptic billiard that has been studied by Delshams and Ramirez Ros in \cite{delshams1996poincare}. Since at $O(\varepsilon + | \delta|)$ the components of $B_{1}$ and $B_{2}$ given by formula (\ref{components}) simply add, the distance between perturbed invariant manifolds $W^{s,u}(\tilde{\Lambda})$ for $\varepsilon, \delta \neq 0$ is given by (\ref{modifieddist}).

\section{Computation of the Melnikov function $M_{1}$ for time-dependent ellipse}

The Melnikov functions (\ref{m1}), (\ref{m2}) can be computed analytically in terms of elliptic functions using the theory developed in \cite{delshams1996poincare} to give the formulas (\ref{melnikovellipse}) and (\ref{melnikdelta}) respectively. In this appendix we will derive the formula (\ref{melnikovellipse}). First, let us introduce some notation following \cite{delshams1996poincare} and quote the Proposition 3.1 we use from \cite{delshams1996poincare}.

Given a parameter $m \in [0,1]$, we have the following complete elliptic integrals of the first and second kind respectively:

$$K=K(m)=\int^{\pi/2}_{0}(1-m \sin(\theta))^{-1/2} d\theta,$$
$$E=E(m)=\int^{\pi/2}_{0}(1-m\sin(\theta))^{1/2} d\theta.$$
The incomplete elliptic integral of the second kind is

$$E(u)=E(u\mid m):=\int^{u}_{0} \dn^{2}(v\mid m) dv,$$
where the function $\dn$ is one of the Jacobian elliptic functions.

Further, $K'=K'(m):=K(1-m)$, $E'=E'(m):=E(1-m)$ and if one of $m, K, K', E, E', \frac{K'}{K}$ is given, all the rest are determined. We determine the parameter $m$ for a given $T, h >0$  by relation $$\frac{K'}{K} = \frac{T}{h}.$$
\
From now on, we do not explicitly write the dependence of $K, K', E, E', m$ on $T$ and $h$. We introduce a function $\chi_{T}(z)$,

$$\chi_{T}(z) = \left(\frac{2K}{h}\right)^{2}\left(\frac{E'}{K'} - 1\right)z + \left(\frac{2K}{h}\right)E\left(\frac{2Kz}{h} + K'i\right),$$
\\
with the properties: (1): $\chi$ is meromorphic on $\mathbb{C}$, (2): $\chi$ is $Ti$-periodic with $h$-periodic derivative, (3): the poles of $\chi$ are in the set $h\mathbb{Z} + Ti\mathbb{Z}$, all simple with residue $1$. The following formula is easily derived using the properties of elliptic functions:

\begin{equation}
\chi(i\pi/2 - \tau) -\chi(h+i\pi/2 - \tau) = -2.
\label{ch}
\end{equation}
It is also easily shown \cite{delshams1996poincare} that the following relation holds:

\begin{equation}
\chi (z +h) - \chi(z) = \frac{2\pi}{T}.
\label{chh}
\end{equation}

For an isolated singularity $z_{0} \in \mathbb{C}$ of a function $q$, we denote  by $a_{j}(q,z_{0})$ the coefficient of $(z-z_{0})^{-j}$ in the Laurent series of $q$ around $z_{0}$.

Then the following result holds. 

\begin{proposition}
 \cite{delshams1996poincare}, \cite{delshams1997melnikov}

Let $q$ be a function satisfying:
\begin{itemize}
\item $q$ is analytic in $\mathbb{R}$, with only isolated singularities in $\mathbb{C}$
\item $q$ is $Ti$-periodic for some $T>0$,
\item $|q(\tau)| \leq A e^{-c|\Re \tau|}$ when $| \Re\tau| \rightarrow \infty$, for some constants $A,c\geq0.$ 
\end{itemize}
Then, $Q(\tau) = \sum_{n = -\infty}^{\infty}  {q(\tau+hn)}$ is analytic in $\mathbb{R}$, has only isolated singularities in $\mathbb{C}$, and is doubly periodic with periods $h \neq 0$, where $h \in \mathbb{R}$ and $Ti$. Furthermore, $Q(\tau)$ may be expressed as

\begin{equation}
Q(\tau) = -\sum_{z \in \mbox{Sing}_{T}(q)}  {\sum_{j\geq0}  {\frac{a_{j+1}(q,z)}{j!}\chi_{T}^{j}(z-\tau)} },
\label{ss}
\end{equation}
\\
where $\mbox{Sing}_{T}(q)$ is the set of singularities of $q$ in $I_{T} = \{z \in \mathbb{C}: 0< \Im z<T\}$, and $\chi^{j}$ denotes the $j$-th derivative of $\chi$.

\label{prop}
\end{proposition}

We note that if $q$ is meromorphic, then $Q(\tau)$ is elliptic and can be computed analytically. Now we proceed to give the derivation of (\ref{melnikovellipse}). We will show that the sum (\ref{m1}) is an elliptic function with two periods $\log \lambda$ and $\pi i$ where $\lambda$ given by (\ref{evalue}), and then apply the above proposition to compute (\ref{m1}).

Let us denote by $I_{\varphi, \theta, E, t}$ the partial derivatives of integral $I$ given by (\ref{integs}) with respect to $\varphi$, $\theta$, $E$ and $t$ respectively. We will evaluate (\ref{m1}) over the unperturbed homoclinic $W_{2}(\tilde{\Lambda})$ for fixed $(E,t)$, hence we use the parameter $\xi = \tan \frac{\varphi_{n}}{2}$ and the equation of $W_{2}(\tilde{\Lambda})$ given by (\ref{w22}). Then the formula (\ref{parameterisation}) expresses $\theta$ as a function of $\varphi$ to give $W_{2}(\Lambda)$ in terms of $\xi$, holding $(E,t)$ fixed at some initial value $(E_{0},t_{0}) = (E_{n}, t_{n})$. Using $B_{1} = (f_{1},f_{2},f_{3},f_{4})^{T}$ from (\ref{components}) and expressing $f_{1,2,3,4}$ in terms of $\xi$ we write (\ref{m1}) as

$$M_{1}:= \sum_{n = -\infty}^{\infty}  { \langle \nabla I(B_{0}(x_{n}),B_{1}(x_{n}) \rangle  } = \sum_{n = -\infty}^{\infty}  {f_{1}(\xi_{n}) I_{\varphi}(\xi_{n+1}) + f_{2}(\xi_{n}) I_{\theta}(\xi_{n}) + f_{4}(\xi_{n}) I_{t}(\xi_{n+1}) } .$$
Note we have written  $\varphi$ and $\theta$ in terms of $\xi$ by virtue of (\ref{parameterisation}), and $x_{n} = (\varphi_{n}, \theta_{n}, E_{n}, t_{n})$ as before. We have suppressed the dependence of functions $f_{1,2,3,4}$ and integral $I$ on $(E,t)$. Note that the term $f_{3}I_{E}$ is identically zero since $I$ is independent of energy  $E$ and thus we omit it from the above sum. Since parametrisation (\ref{parameterisation}) yields $\xi_{n+1} = \lambda^{-1} \xi_{n}$, we may express the sum above purely in terms of $\xi_{n}$, i.e.

\begin{equation}
M_{1} = \sum_{n = -\infty}^{\infty}  {f_{1}(\xi_{n}) I_{\varphi}(\lambda^{-1}\xi_{n}) + f_{2}(\xi_{n}) I_{\theta}(\lambda^{-1}\xi_{n}) + f_{4}(\xi_{n}) I_{t}(\lambda^{-1}\xi_{n}) } = \sum_{n = -\infty}^{\infty}  {F(\xi_{n}) },  
\label{mtermsxi}
\end{equation}
for certain function $F$. Introduce a change of variables $\tau$ defined by $\e^{\tau}=\xi$. Since $\xi \in (0, \infty)$, then $\tau \in (-\infty, \infty)$, and for brevity put $h = \log \lambda$, as in \cite{delshams1996poincare}. Then $M_{1}$ becomes $M_{1} = \sum F(e^{\tau + nh}) = \sum \tilde{F}(\tau + nh) $, after swapping $n \mapsto -n$. Here $\tilde{F}(\tau) = f_{1}(\tau)I_{\varphi}(\tau) +f_{2}(\tau)I_{\theta}(\tau) + f_{4}(\tau)I_{t}(\tau)$.  Using the formulae (\ref{euclidzero}), (\ref{f1}), (\ref{f2}), (\ref{f3}), (\ref{integs}) together with (\ref{parameterisation}), we obtain the following expressions for $f_{1}I_{\varphi}|_{W_{2}}$, $f_{2}I_{\theta}|_{W_{2}}$ and $f_{4}I_{t}|_{W_{2}}$ in terms of $\tau$:

\begin{equation}
f_{1}(\tau)I_{\varphi}(\tau)|_{W_{2}} = \frac{-8 \dot{a}c^{2}\lambda(\lambda+1)\e^{2\tau}(\lambda^{2} - \e^{2\tau})^{2}}{v_{0}(1+\e^{2\tau})(\e^{2\tau} + \lambda^{3})(\e^{2\tau}+\lambda^{2})^{2}} + \frac{16\dot{b}ac^{2}\lambda^{2}\e^{2\tau}(\lambda-\e^{2\tau})(\lambda^{2}-\e^{2\tau})}{bv_{0}(1+\e^{2\tau})(\e^{2\tau}+\lambda^{3})(\e^{2\tau}+\lambda^{2})^{2}},
\label{f1Ipar}
\end{equation}
\
\begin{equation}
\begin{split}
f_{2}(\tau)I_{\theta}(\tau )|_{W_{2}} & =\underbrace{-\frac{16a^{2}c\lambda^{2}(\dot{a}ba^{-1}-\dot{b})\e^{2\tau}(\lambda^{2}-\e^{2\tau})(\lambda+\e^{2\tau})}{v_{0}b(\lambda^{3}+\e^{2\tau})(\lambda^{2}+\e^{2\tau})^{2}(\e^{2\tau}+1)}}_{1} + \underbrace{\frac{8ac\lambda\e^{2\tau}(-\dot{a}(\lambda+1)(\lambda^{2}-\e^{2\tau}) +2\lambda\dot{b}ab^{-1}(\lambda - \e^{2\tau})) }{v_{0}(\e^{2\tau}+\lambda^{3})(\e^{2\tau}+\lambda^{2})(\e^{2\tau}+1)}}_{2} \\
& + \underbrace{\frac{16c^{2}\lambda^{2}\e^{2\tau}(\dot{a}b(\lambda^{2}-\e^{2\tau})^{2} + 4\dot{b}a\lambda^{2}\e^{2\tau})}{v_{0}b(\e^{2\tau}+\lambda^{3})(\e^{2\tau}+\lambda^{2})^{2}(\e^{2\tau}+\lambda)}}_{3},
\end{split}
\label{f2Ithe}
\end{equation}
\
\begin{equation}
f_{4}(\tau) I_{t}(\tau)|_{W_{2}} = \frac{16a^{2}\lambda^{2}\e^{2\tau}(a\dot{b}b^{-1} - \dot{a})(\e^{2\tau}+\lambda)}{(\e^{2\tau}+\lambda^{3})(\e^{2\tau}+\lambda^{2})(\e^{2\tau}+1)}.
\label{f4It}
\end{equation}
\ 
Since in this appendix we are only considering the orbit of map $B$ restricted to $W_{2}$, we drop the subscript $W_{2}$ for brevity. It is clear that $f_{1}I_{\varphi}$, $f_{2}I_{\theta}$, $f_{4}I_{t}$ are analytic on $\mathbb{R}$, exponentially decay at infinity  and are $\pi i$-periodic on $\mathbb{C}$ with isolated  singularities that are poles, hence they are meromorphic, satisfying Proposition~\ref{prop}. Therefore we may apply Proposition~\ref{prop} to compute  $M_{1} = \sum {F(\tau)}$. We therefore take $T = \pi$,  $\frac{K'}{K} = \frac{\pi}{h}$ and $\chi_{\pi}(z) = \chi(z)$.

In the following computations, we will be using the following formula that can be shown using properties of elliptic functions:

\begin{equation}
\chi(z + lh + h/2) - \chi(z+lh) =1+ \frac{2K}{h}m \sn\left(\frac{2K\tau}{h}\right) \cd\left(\frac{2K\tau}{h}\right) = 1 + Y(\tau),
\label{chtrue}
\end{equation}
where $l \in \mathbb{Z}$, and we defined $Y(\tau) = \frac{2K}{h}m \sn\left(\frac{2K\tau}{h}\right) \cd\left(\frac{2K\tau}{h}\right)$ for brevity. Also, we have

\begin{equation}
\chi^{'}(\frac{\pi i}{2} - \tau) =\left(\frac{2K}{h}\right)^{2}\left( \frac{E'}{K'} - 1 +\dn^{2}\left(\frac{2K\tau}{h}\right)  \right) = X(\tau),
\label{jac2}
\end{equation}
where for brevity we put $X(\tau) =\left(\frac{2K}{h}\right)^{2}\left( \frac{E'}{K'} - 1 +\dn^{2}\left(\frac{2K\tau}{h}\right)  \right)$ and ' denotes differentiation w.r.t. $\tau$. In (\ref{chtrue}) and (\ref{jac2}) we have used the identities $\dn(u) = dn(-u)$, $\dn(u) = -\dn(u + 2K + 2K'i)$, $E(-u) = -E(u)$, $E(u+2K + 2K'i) = E(u) + 2E + 2i(K' - E')$, $E(K - u) = E - E(u) + m \sn(u) \cd(u)$ and the Legendre equality $EK' + E'K - K'K = \frac{\pi}{2}$ together with the formulas (\ref{ch}), (\ref{chh}).
 
We will compute separately the three components $\sum f_{1}I_{\varphi}$, $\sum f_{1} I_{\theta}$ and $\sum f_{4}I_{t}$ of the  expression for $M_{1}$ given by (\ref{mtermsxi}). Further, wherever it facilitates the computations, will consider individually the $\dot{a}$ and $\dot{b}$ components of the sums  $\sum f_{1}I_{\varphi}$, $\sum f_{1} I_{\theta}$ and $f_{4}I_{t}$.

Consider  $\sum f_{1}I_{\varphi}$. Take the $\dot{a}$ component of $\sum f_{1}I_{\varphi}$ with the coefficient $\frac{-8c^{2}\lambda(\lambda+1)}{v_{0}}$ factored out. Denote by $ a_{j}(f_{1}I_{\varphi}, z_{0}; \dot{a})$ the corresponding Laurent series coefficient of the $\dot{a}$ component of $ f_{1}I_{\varphi}$ at $z_{0}$ (here  $\frac{-8c^{2}\lambda(\lambda+1)}{v_{0}}$ is factored out).  We have:

\begin{itemize}
\item  Simple pole at $z= i\pi/2$ with $a_{1}(f_{1}I_{\varphi}, i\pi/2; \dot{a})=\frac{1+2\lambda + \lambda^{4}}{2(-1+\lambda)^{3}(1+\lambda)^{2}(1+\lambda + \lambda^{2})},$
\item  simple pole at $z = 3h/2 + i\pi/2$  with  $a_{1}(f_{1}I_{\varphi}, 3h /2 + i\pi/2; \dot{a})=\frac{-1-2\lambda - \lambda^{2}}{2(-1+\lambda)^{3}(1+\lambda + \lambda^{2})},$
\item double pole at $z = h + i\pi/2$ with $a_{1}(f_{1}I_{\varphi}, h + i\pi/2; \dot{a}) = \frac{2\lambda}{(-1+\lambda)^{3}(1+\lambda)^{2}}$ and $a_{2}(f_{1}I_{\varphi}, h + i\pi/2; \dot{a}) = \frac{1}{(\lambda - 1)^{2}(1+\lambda)}.$
\end{itemize}

Observe that $a_{1}(f_{1}I_{\varphi}, i\pi/2; \dot{a}) + a_{1}(f_{1}I_{\varphi}, 3 h/2 + i\pi/2; \dot{a}) = -a_{1}(f_{1}I_{\varphi}, h+i\pi/2; \dot{a})$. Thus we obtain the following formula for $\dot{a}$ component of  $\sum f_{1}I_{\varphi}$:

$$\sum f_{1} I_{\varphi}|_{\dot{a}} =\frac{8c^{2}\lambda(\lambda+1)}{v_{0}} \left(  -2a_{1}(i\pi/2) + a_{1}(3h/2 + i\pi/2)\left(1 + Y(\tau) \right) + a_{2}(h + i\pi/2)X(\tau) \right).$$
\
Similarly, considering the $\dot{b}$ component of  $\sum f_{1}I_{\varphi}$ (see (\ref{f1Ipar})), with factored out $\frac{16\dot{b}ac^{2}\lambda^{2}}{bv_{0}}$, we have:

\begin{itemize}
\item Simple pole at $z = \frac{i \pi}{2}$ with $a_{1}(f_{1}I_{\varphi},i\pi/2; \dot{b}) = \frac{\lambda^{2} +1}{2(\lambda+1)(\lambda-1)^{3}(\lambda^{2} + \lambda+1)},$
\item simple pole at $z = 3h/2 + i\pi/2$ with  $a_{1}(f_{1}I_{\varphi}, 3h/2 + i\pi/2; \dot{b}) = \frac{-(\lambda+1)(\lambda^{2}+1)}{2\lambda(\lambda-1)^{3}(\lambda^{2}+\lambda+1)},$
\item double pole at $h + i\pi/2$ with $a_{1}(f_{1}I_{\varphi}, h + i\pi/2; \dot{b}) = \frac{\lambda^{2}+1}{2\lambda(\lambda+1)(\lambda-1)^{3}}$ and $a_{2}(f_{1}I_{\varphi}, h + i\pi/2; \dot{b})  = \frac{1}{2\lambda(\lambda-1)^{2}}.$
\end{itemize}
Here $a_{1}(f_{1}I_{\varphi}, i\pi/2; \dot{b}) + a_{1}(f_{1}I_{\varphi}, 3h/2 + i\pi/2; \dot{b}) = -a_{1}(f_{1}I_{\varphi}, h + i\pi/2; \dot{b})$. Thus we have the following formula for $\dot{b}$ component of $\sum f_{1} I_{\varphi}$:

$$\sum f_{1} I_{\varphi}|_{\dot{b}} = -\frac{16\dot{b}ac^{2}\lambda^{2}}{bv_{0}} \left(-2a_{1}(i\pi/2) + a_{1}(3h/2 + i\pi/2)(1+ Y(\tau)) + a_{2}(h + i\pi/2)X(\tau)\right).$$
\
Adding the $\dot{a}$ and $\dot{b}$ contributions and simplifying gives the formula

\begin{equation}
\begin{aligned}
\sum { f_{1}I_{\varphi} } & = \frac{\dot{a}}{v}\left(\frac{-2b^{2}(3a^{4}+2a^{2}c^{2}+c^{4})}{ac(3a^{2} + c^{2})} + 2b^{2}X(\tau) - \frac{4a^{3}b^{2}}{c(3a^{2}+c^{2})}Y(\tau) \right) \\
& + \frac{\dot{b}}{bv}\left(\frac{2b^{2}(3a^{2}-c^{2})(a^{2}+c^{2})}{c(3a^{2}+c^{2})} - 2ab^{2}X(\tau) + \frac{4a^{2}b^{2}(a^{2}+c^{2})}{c(3a^{2}+c^{2})}Y(\tau)\right).
\end{aligned}
\label{f1i}
\end{equation}

Next we will compute $\sum f_{2}I_{\theta}$; the corresponding expression (\ref{f2Ithe}) consist of the sum of three parts. Take the part $(1)$ defined by braces, and factor out $\frac{-16a^{2}c \lambda^{2}(\dot{a}ba^{-1} - \dot{b})}{v_{0}b}$. Denote by $ a_{j}(f_{1}I_{\theta}, z_{0}; (1))$ the corresponding Laurent series coefficient of the $(1)$ component of $ f_{1}I_{\varphi}$ at $z_{0}$, (without coefficient $\frac{-16a^{2}c \lambda^{2}(\dot{a}ba^{-1} - \dot{b})}{v_{0}b}$). We have

\begin{itemize}
\item  Simple pole at $z = i\pi/2$ with $a_{1}(f_{2}I_{\theta}, i\pi/2; (1)) = \frac{\lambda^{2}+1}{2(\lambda^{2} - 1)^{2} (\lambda^{2} + \lambda + 1)}$,
\item  simple pole at $z = 3h/2 + i\pi/2$ with $a_{1}(f_{2}I_{\theta}, 3h/2 + i \pi/2; (1)) = \frac{(\lambda +1)^{2}}{2\lambda (\lambda-1)^{2}(\lambda^{2} + \lambda + 1)}$,
\item  double pole at $z = h + i\pi/2$ with $a_{1}(f_{2}I_{\theta}, h + i\pi/2; (1)) = \frac{-1-4 \lambda - \lambda^{2}}{2\lambda(\lambda^{2} - 1)^{2}}$ and $a_{2}(f_{2}I_{\theta}, h + i\pi/2; (1) = \frac{1}{2\lambda(\lambda^{2}-1)}$,
\end{itemize}
\
where $a_{1}(f_{2}I_{\theta}, i\pi/2; (1)) + a_{1}(f_{2}I_{\theta}, 3h/2 + i\pi/2; (1)) = -a_{1}(f_{2}I_{\theta}, h + i\pi/2; (1))$. Hence the sum for the part (1) is:

\begin{equation}
\begin{aligned}
\sum f_{2} I_{\theta}|_{(1)} &  = \frac{16a^{2}c\lambda^{2}(\dot{a}b a^{-1} - \dot{b})}{v_{0}b} \left( -2a_{1}(f_{2}I_{\theta}, i\pi/2; (1) ) + a_{1}(f_{2}I_{\theta}, 3h/2 + i\pi/2; (1) ) \left(1 + Y(\tau) \right)\right) \\ 
& +  \frac{16a^{2}c\lambda^{2}(\dot{a}b a^{-1} - \dot{b})}{v_{0}b} (a_{2}(f_{2}I_{\theta}, h + i\pi/2; (1)) X(\tau)).
\end{aligned}
\label{f2sum1}
\end{equation}
\
Let us consider the $\dot{a}$ and $\dot{b}$ components of the second component $(2)$ of (\ref{f2Ithe}) separately. Take contribution in $\dot{a}$. Factor out $\frac{-8a\dot{a}c\lambda(\lambda+1)}{v_{0}}$. Denote by $ a_{j}(f_{1}I_{\theta}, z_{0}; \dot{a}, (2))$ the corresponding Laurent series coefficient of the $\dot{a}$ coefficient of the $(2)$ component of $ f_{1}I_{\varphi}$ at $z_{0}$. We have

\begin{itemize}
\item Simple pole at $z = \frac{i\pi}{2}$ with $a_{1}(f_{2}I_{\theta}, i\pi/2; \dot{a}, (2)) = \frac{\lambda^{2}+1}{2(\lambda^{2}-1)(\lambda-1)(\lambda^{2} + \lambda+1)}$,
\item simple pole at $z = 3h/2 + i\pi/2$ with $a_{1}(f_{2}I_{\theta}, 3h/2 + i\pi/2; \dot{a}, (2)) = \frac{\lambda+1}{2(\lambda-1)^{2} (\lambda^{2}+\lambda+1)}$,
\item Simple pole at $z = h + i\pi/2$ with $a_{1}(f_{2}I_{\theta}, h + i\pi/2; \dot{a}, (2)) = \frac{-1}{(\lambda-1)^{2} (\lambda+1)}$,
\end{itemize}
\
where $a_{1}(f_{2}I_{\theta}, i\pi/2; \dot{a}, (2)) + a_{1}(f_{2}I_{\theta}, 3h + i\pi/2; \dot{a}, (2)) = -a_{1}(f_{2}I_{\theta}, h + i\pi/2; \dot{a}, (2))$. Thus the corresponding sum for part $(2)$ for $\dot{a}$ contribution is

\begin{equation}
\sum f_{2} I_{\theta}|_{\dot{a}, (2)} =  \frac{8a\dot{a}c\lambda(\lambda+1)}{v_{0}} \left(-2a_{1}(f_{2}I_{\theta}, i\pi/2; \dot{a}) + a_{1}(f_{2}I_{\theta}, 3h/2 + i\pi/2; \dot{a})(1+ Y(\tau)) \right).
\label{f2sum2a}
\end{equation}
\
For $\dot{b}$ component of $(2)$ of (\ref{f2Ithe}), first factor out $\frac{16a^{2}c\lambda^{2}\dot{b}}{bv}$. Using the notation as above, we find:

\begin{itemize}
\item Simple pole at $z = i\pi/2$ with $a_{1}(f_{2}I_{\theta}, i\pi/2; \dot{b}, (2)) = \frac{1}{2(\lambda-1)^{2}(\lambda^{2} + \lambda + 1)}$,
\item simple pole at $z = 3h/2 + i\pi/2$ with $a_{1}(f_{2}I_{\theta}, 3h/2 + i\pi/2; \dot{b}, (2)) = \frac{\lambda^{2} +1}{2\lambda(\lambda-1)^{2}(\lambda^{2} + \lambda+1)}$,
\item simple pole at $z = h +  i\pi/2$ with $a_{1}(f_{2}I_{\theta}, h + i\pi/2; \dot{b}, (2)) = \frac{-1}{2\lambda(\lambda-1)^{2}}$,
\end{itemize}
\
where $a_{1}(f_{2}I_{\theta}, i\pi/2; \dot{b}, (2)) + a_{1}(f_{2}I_{\theta}, 3h/2 + i\pi/2; \dot{b}, (2)) = -a_{1}(f_{2}I_{\theta}, h + i\pi/2; \dot{b}, (2))$. Thus the second component (2) sum for $\dot{b}$ is

\begin{equation}
\sum f_{2} I_{\theta}|_{\dot{b},(2)} =  -\frac{16a^{2}c\lambda^{2}\dot{b}}{bv} \left(-2a_{1}(f_{2}I_{\theta}, i\pi/2; \dot{b}) + a_{1}(f_{2}I_{\theta}, 3h/2 + i\pi/2; \dot{b})(1+ Y(\tau)) \right).
\label{f2sum2b}
\end{equation}
\
Finally, let us consider the $\dot{a}$ and $\dot{b}$ coefficients of part $(3)$ of (\ref{f2Ithe}) separately. For $\dot{a}$ component of (3), we factor out $\frac{16\dot{a}c^{2}\lambda^{2}}{v_{0}}$ and obtain

\begin{itemize}
\item  Simple pole at $z = 3h/2 + i\pi/2$ with $a_{1}(f_{2}I_{\theta}, 3h/2 + i\pi/2; \dot{a}, (3)) = -a_{1}(h/2 + i\pi/2) = \frac{-1 - \lambda}{2\lambda(\lambda-1)^{3}}$,
\item  Simple pole at $z = h + i\pi/2$ with $a_{1}(f_{2}I_{\theta}, h + i\pi/2; \dot{a}, (3)) = 0$, $a_{2}(h + i\pi/2) = \frac{1}{\lambda(\lambda-1)^{2}}$.
\end{itemize}

Hence we get the corresponding sum 

\begin{equation}
\sum f_{2} I_{\theta}|_{\dot{a},(3)} =  -\frac{16\dot{a}c^{2}\lambda^{2}}{v_{0}} \left(2a_{1}(f_{2}I_{\theta}, 3h/2 + i\pi/2; \dot{a}, (3)) + a_{2}(f_{2}I_{\theta}, h + i\pi/2; \dot{a}, (3)) X(\tau) \right).
\label{f2sum3a}
\end{equation}
\
For third component $(3)$ coefficient in $b$, we factor out $\frac{64a\dot{b}c^{2}\lambda^{4}}{bv_{0}}$ to yield

\begin{itemize}
\item Simple pole at $z = 3h/2 + i\pi/2$ with $a_{1}(f_{2}I_{\theta}, 3h/2 + i\pi/2; \dot{b}, (3) ) = -a_{1}(h/2 + i\pi/2)= \frac{1}{2\lambda^{2}(\lambda+1)(\lambda-1)^{3}}$,
\item simple pole at $z = h + i\pi/2$ with $a_{1}(f_{2}I_{\theta}, h + i\pi/2; \dot{b}, (3)) = 0$, $a_{2}(f_{2}I_{\theta}, h+ i\pi/2; \dot{b}, (3)) = \frac{-1}{4\lambda^{3}(\lambda-1)^{2}}$,
\end{itemize}
\
which gives the sum 

\begin{equation}
\sum f_{2} I_{\theta}|_{\dot{b},(3)} =  -\frac{64a\dot{b}c^{2}\lambda^{4}}{bv_{0}} \left(2a_{1}(f_{2}I_{\theta}, 3h/2 + i\pi/2; \dot{b}, (3)) + a_{2}(f_{2}I_{\theta}, h+ i\pi/2; \dot{b}, (3))X(\tau) \right).
\label{f2sum3b}
\end{equation}
\
Adding the expressions (\ref{f2sum1}), (\ref{f2sum2a}), (\ref{f2sum2b}), (\ref{f2sum3a}) and (\ref{f2sum3b})  and simplifying gives the formula

\begin{equation} 
\sum { f_{2}I_{\theta} } = \frac{\dot{a}}{v}\left(\frac{2b^{2}(9a^{4}+c^{4})}{ac(3a^{2}+c^{2})} - 6b^{2}X(\tau) + \frac{12a^{3}b^{2}}{c(3a^{2}+c^{2})}Y(\tau)\right) + \frac{\dot{b}}{bv}\left(6ab^{2} X(\tau) - \frac{2b^{2}(3a^{2}-c^{2})}{c} - \frac{4a^{2}b^{2}}{c}Y(\tau)\right).
\label{sumf2i}
\end{equation}

Finally we compute $\sum_{n = -\infty}^{\infty}  { f_{4}I_{t}}$. Let us take out the factor $16a^{2}\lambda^{2}(a\dot{b}b^{-1} - \dot{a})$. Then we obtain

\begin{itemize}
\item Simple pole at $z = i\pi/2$ with $a_{1}(f_{4}I_{t}, i\pi/2)=\frac{1}{2(-1-\lambda + \lambda^{3} + \lambda^{4})}$,
\item simple pole at $z = 3h/2 + i\pi/2$ with  $a_{1}(f_{4}I_{t}, 3\log \lambda /2 + i\pi/2)=\frac{-1-\lambda}{2\lambda(-1+\lambda^{3})}$,
\item simple pole at $z = h + i\pi/2$ with  $a_{1}(f_{4}I_{t}, \log \lambda + i\pi/2) = \frac{1}{2\lambda(-1+\lambda^{2})}.$
\end{itemize}

Observe that $a_{1}(f_{4}I_{t}, \log \lambda +i\pi/2) + a_{1}(f_{4}I_{t}, 3 \log \lambda/2 + i\pi/2) = -a_{1}(f_{4}I_{t}, i\pi/2)$. Therefore the sum is

$$\sum {f_{4}I_{t} } = -  \left(\frac{a\dot{b}}{bv} - \frac{\dot{a}}{v}\right)\left(16a^{2} \lambda^{2}\right) \left(a_{1}(3h/2 + i\pi/2)(3 + Y(\tau)) + 2a_{1}(h + i\pi/2)\right).$$

\
Upon simplifying, this yields

\begin{equation}
\sum {f_{4}I_{t} } = \left(\frac{a\dot{b}}{bv} - \frac{\dot{a}}{v}\right)\left(\frac{4ab^{2}(3a^{2}-c^{2})}{c(3a^{2}+c^{2})} + \frac{8a^{3}b^{2}}{c(3a^{2}+c^{2})}Y(\tau)\right).
\label{sumf4i}
\end{equation}

Finally adding (\ref{f1i}), (\ref{sumf2i}) and (\ref{sumf4i}) yields (\ref{melnikovellipse}).

\section{Computation of zero order distance $D_{0}$}

In this appendix we show that the free flight distance (\ref{euclidzero}) that is incorporated in the formulas for $f_{1}, f_{2}, f_{4}$ turns out to be an elliptic function when evaluated on the homoclinic manifold $W_{2}(\tilde{\Lambda})$ (see formula (\ref{w22})).

\begin{proposition}
The zero order in $\varepsilon, \delta$ free flight distance $D_{0}$ given by equation (\ref{euclidzero}) evaluated on $W_{2}(\tilde{\Lambda})$ expressed in terms of $\xi$ is

\begin{equation}
D_{0}|_{W_{2}(\tilde{\Lambda})}=\frac{2a(\lambda +\xi^{2}_{n})^{2}}{(\lambda^{2}+\xi^{2}_{n})(1+ \xi^{2}_{n})} = \frac{2a(1+\lambda \xi^{2}_{n+1})^{2}}{(1+\xi^{2}_{n+1})(1+\lambda^{2}\xi^{2}_{n+1})}.
\label{dist-c}
\end{equation}

\label{prop2}
\end{proposition}

\begin{proof}
Recall that $W_{2}(\tilde{\Lambda})$ corresponds to the billiard orbit passing the focus at $(-c,0)$, and therefore each consecutive collision point on the boundary (\ref{ellipse}) is in alternate halves of the ellipse. In terms of variable $\varphi \Mod{\pi}$  the negative signs in the square brackets in (\ref{euclidzero}) change to positive and gives

$$D_{0}|_{W_{2}(\tilde{\Lambda})}=\sqrt{a^{2}[\cos(\varphi_{n}) + \cos(\varphi_{n+1})]^{2} + b^{2}[\sin(\varphi_{n}) + \sin(\varphi_{n+1})]^2}.$$
\\
Taking the parametrisation (\ref{parameterisation}) to express trigonometric functions of $\varphi_{n+1}$ and $\varphi_{n}$ in terms of $\xi_{n}$ and using the identity $\frac{b^2}{a^2}(1+\lambda)^2 \equiv 4\lambda$, with $\lambda$ given by (\ref{evalue}), we obtain:

$$D_{0}|_{W_{2}(\tilde{\Lambda})} = \sqrt{ a^{2} \left( \frac{4\lambda}{(1+\lambda)^{2}}\left( \frac{2\xi_{n}}{1 + \xi^{2}_{n}}+ \frac{2\lambda \xi_{n}}{\lambda^{2} + \xi^{2}_{n}} \right) ^{2} + \left( \frac{1 - \xi^{2}_{n}}{1 + \xi^{2}_{n}} + \frac{\lambda^{2} - \xi^{2}_{n}}{\lambda^{2} + \xi^{2}_{n}} \right)^{2} \right)}.$$
Taking a positive root of this expression since it is Euclidean distance and simplifying yields:

$$D_{0}|_{W_{2}(\tilde{\Lambda})}  = \sqrt{ a^{2}\left(\frac{4(\lambda + \xi_{n}^{2})^{4}}{(1+\xi_{n}^{2})^{2}(\lambda^{2}+ \xi_{n}^{2})^{2}}\right)}.$$
\end{proof}

\begin{remark}
Note that in the limit of $n \rightarrow \infty$ , (\ref{dist-c}) yields $D_{0}|_{W_{2}(\tilde{\Lambda})} =2a$, i.e. the length of the major axis of the ellipse,  which corresponds the hyperbolic fixed point $z$.
\end{remark}

\begin{remark}
Analogously we may calculate $D_{0}|_{W_{1}(\tilde{\Lambda})}$, which is:

\begin{equation}
D_{0}|_{W_{1}(\tilde{\Lambda})}=\frac{2a(1+\lambda\xi^{2}_{n})^{2}}{(1+\xi^{2}_{n})(1+\lambda^{2} \xi^{2}_{n})} =\frac{2a(\lambda +\xi^{2}_{n+1})^{2}}{(\lambda^{2}+\xi^{2}_{n+1})(1+ \xi^{2}_{n+1})}.
\label{distc}
\end{equation}
\end{remark}

\section{Scattering map computations}

We provide a derivation of truncated scattering map formula (\ref{Efib}). From the $E$ component of $S_{\Gamma}$ given by (\ref{scatteringE}), we need to compute the sum

$$\sum_{n=-\infty}^{\infty}  f_{3}(\cdot, E_{n}, t_{n})|_{\Lambda} - f_{3}(\varphi_{n}, \theta_{n}(\varphi_{n}), E_{n}, t_{n})|_{W_{2}(\bar{\Lambda})} = 2\sqrt{2E_{0}} \sum_{n = -\infty}^{\infty}  {- \dot{a}(t_{n+1}) + u_{n+1} \sin (\theta_{n+1}) },$$
where $u_{n+1} = \frac{\dot{a}b\cos^{2}(\varphi_{n+1}) + a\dot{b}\sin^{2}(\varphi_{n+1})}{\sqrt{ a^{2}\sin^{2}(\varphi_{n+1}) + b^{2}\cos^{2}(\varphi_{n+1})}}$ is the normal boundary speed at boundary point $\varphi_{n+1}$ evaluated on the homoclinic manifold $W_{2}(\tilde{\Lambda})$. Rewriting $\varphi$ in terms of $\xi$ using (\ref{parameterisation}) gives

$$u_{n+1} =  \frac{\dot{a}b(1-\xi^{2}_{n+1})^{2} + 4a\dot{b}\xi^{2}_{n+1}}{(1+\xi^{2}_{n+1})\sqrt{4c^{2}\xi^{2}_{n+1} + b^{2}(1+\xi^2_{n+1})^2}}.$$
\
Now $$4c^{2} \xi^{2}_{n+1} + b^{2}(1+\xi^{2}_{n+1})^2 \equiv \frac{b^{2}(\xi^{2}_{n+1}+\lambda)(\lambda \xi^{2}_{n+1}+1)}{\lambda},$$
\
and expressing $\sin(\theta)$ in terms of $\xi$ using (\ref{parameterisation}), we obtain

$$u_{n+1} \sin (\theta_{n+1}) - \dot{a}(t_{n+1}) = \frac{\lambda(\dot{a}b(1-\xi_{n+1}^{2})^{2}+4a\dot{b}\xi_{n+1}^{2})}{b(\xi^{2}_{n+1} + \lambda)(\lambda \xi^{2}_{n+1}+1)} - \dot{a}(t_{n+1}) = \frac{\dot{a} \xi^{2}_{n+1}(\lambda+1)^{2}}{(\xi^{2}_{n+1}+\lambda)(\lambda \xi^{2}_{n+1}+1)} + \frac{4a\dot{b}\lambda \xi^{2}_{n+1}}{b(\xi^{2}_{n+1}+\lambda)(\lambda \xi^{2}_{n+1}+1)}.$$

It is clear that we need to calculate $\sum_{n = -\infty}^{\infty} {\frac{\xi^{2}_{n}}{(\xi^{2}_{n}+\lambda)(\lambda \xi^{2}_{n}+1)} }$. After a change of variables $\xi = \e^{\tau}, h = \log \lambda$ as before, the function $\frac{\e^{2\tau}}{(\e^{2\tau}+\lambda)(\lambda \e^{2\tau}+1)}$ is bounded for $\Re\, \tau \rightarrow \infty$, is periodic in the complex plane with period $i\pi$, and it has simple poles at $\tau = \pm \log \lambda /2 + i\pi / 2$ with residues $\frac{\mp 1}{2(\lambda^{2}-1)}$, and by using the result $\chi(\frac{\log \lambda}{2} +i\pi/2 - \tau) -\chi(\frac{\log \lambda}{2}+i\pi/2 - \tau) = 2$, we have that $ \sum_{n = -\infty}^{\infty} {\frac{\xi^{2}_{n}}{(\xi^{2}_{n}+\lambda)(\lambda \xi^{2}_{n}+1)} } = \frac{1}{\lambda^{2}-1}.$ Therefore after some manipulations we get

$$u_{n+1} \sin (\theta_{n+1}) - \dot{a}(t_{n+1}) = \frac{b\dot{b}-a\dot{a}}{c},$$
which gives (\ref{Efib}).

The $t$ component of $S_{\Gamma}$ is computed using the geometric property of the ellipse, as given in Section 4.3.

\section{Computation of Melnikov function $M_{2}$}

As remarked previously, $M_{2}$ corresponds to the Melnikov function for the $O(\delta)$ quartic polynomial perturbation of the ellipse, evaluated for the 'fixed' boundary at time $t = t_{n}$. Therefore $M_{2}$ may be evaluated by standard methods developed for static perturbations of elliptic billiards using the generating function approach \cite{delshams1996poincare}. The parametrisation (\ref{ellipse}) of the boundary at fixed time $t_{n}$ for first order in $\delta$ coincides with the parametrisation in \cite{delshams1996poincare} given by equation (4.6) up to the adjustment of a constant factor. Thus, it is easy to essentially re-write the derivation of the Melnikov function in \cite{delshams1996poincare} to obtain 

\begin{equation}
M_{2} = -4m \frac{ab^{2}}{c^2} \left(\frac{2K}{h}\right)^{3} \dn \left(\frac{2K\tau}{h}\right) \sn \left(\frac{2K\tau}{h}\right) \cn \left(\frac{2K\tau}{h}\right).
\label{delshams}
\end{equation}
\
Notice that (\ref{delshams}) coincides with the Melnikov function for the perturbed ellipse in \cite{delshams1996poincare}, (see section 4.4.2 of \cite{delshams1996poincare}), upon setting $c^{2} = 1$, $c_{1}=b$.

\end{document}